\PassOptionsToPackage{cmyk}{xcolor}
\documentclass[11pt,twoside,a4paper,reqno]{amsart}
\usepackage{euscript}
\usepackage{a4wide}
\usepackage[T2A]{fontenc}
\usepackage[utf8]{inputenc}
\usepackage[english]{babel}
\usepackage{amsfonts}
\usepackage{amssymb, amsthm}
\usepackage{amsmath}
\usepackage{graphicx}
\usepackage{geometry}
\usepackage[inline]{enumitem}
\usepackage{changepage}

\usepackage[normalem]{ulem}  % для зачекивания текста
\newcommand{\df}{\mathcal S}
\newcommand{\dfa}{\mathcal \df_{\alpha}}
\newcommand{\ea}{\mathcal \eps_{\alpha}}
\newcommand{\I}{\mathcal{I}}
\newcommand{\osc}{\mathcal O}
\numberwithin{equation}{section}
\RequirePackage{tikz}
\usetikzlibrary{shapes.geometric}
\usetikzlibrary {patterns}
\usepackage{pgfplots}
\pgfplotsset{compat=1.18}

\definecolor{myred}{rgb}{0.8500, 0.3250, 0.0980}

\definecolor{myblue}{rgb}{0, 0.4470, 0.7410}

\makeatletter
\newcommand{\leqnomode}{\tagsleft@true\let\veqno\@@leqno}
\makeatother

\definecolor{webgreen}{rgb}{0,.5,0}
\definecolor{webbrown}{rgb}{.6,0,0}
\definecolor{RoyalBlue}{cmyk}{1, 0.50, 0, 0}

\usepackage[colorlinks=true, breaklinks=true, urlcolor=webbrown, linkcolor=RoyalBlue, citecolor=webgreen,backref=page]{hyperref}
\title{Krein systems with oscillating potentials}
\author{Pavel Gubkin}

\newtheorem{theo}{Theorem}[section]

\newtheorem{Thm}[theo]{Theorem}

\newtheorem{Lem}[theo]{Lemma}
\newtheorem{Prop}[theo]{Proposition}
\newtheorem{Cor}[theo]{Corollary}

\newtheorem{Rema}[theo]{Remark}

\renewcommand{\le}{\leqslant}
\renewcommand{\ge}{\geqslant}

\newcommand{\var}{\mathbf{D}}
\newcommand{\dm}{\mathop{d\mathrm{m}}\nolimits}
\newcommand{\mm}{\mathop{\mathrm{m}}\nolimits}

\def\norm[#1]{\left\| #1 \right\|}
\newcommand\smatrix[4]{\left( \begin{smallmatrix} #1 & #2 \\ #3 &  #4\end{smallmatrix} \right)}

\def\ls{\lesssim}
\def\gs{\gtrsim}
\newcommand{\Szego}{Szeg\H{o} }
\newcommand{\GrBell}{Gr\"{o}nwall–Bellman }
\newcommand{\bm}{\mathbf{m}}

\newcommand{\supp}{\mathop{\mathrm{supp}}\nolimits}

\newcommand{\trace}{\mathop{\mathrm{trace}}\nolimits}

\def\ol{\overline}
\renewcommand{\phi}{\varphi}
\newcommand{\eps}{\varepsilon}

\newcommand{\R}{\mathbb R}

\newcommand{\K}{\mathcal K}

\newcommand{\D}{\mathbb D}
\newcommand{\T}{\mathbb T}

\newcommand{\F}{\mathcal F}

\newcommand{\Cm}{\mathbb C}

\newcommand{\loc}{\rm{loc}}

\renewcommand{\Re}{\mathop{\rm Re}}
\renewcommand{\Im}{\mathop{\rm Im}}

\address{
	\begin{flushleft}
		Pavel Gubkin: gubkinpavel@pdmi.ras.ru, gubkin.pv@yandex.ru \\\vspace{0.1cm}
		St. Petersburg Department of Steklov Mathematical Institute\\
		Russian Academy of Sciences\\
		Fontanka 27, 191023 St. Petersburg, Russia\\\vspace{0.1cm}
		St. Petersburg State University \\
		Universitetskaya nab. 7-9, St. Petersburg, 199034, Russia
	\end{flushleft}
}

\thanks{The work is supported by the Russian Science Foundation grant RScF 23-11-00171, \url{https://rscf.ru/project/23-11-00171/},  and in part by the M\"obius Contest Foundation for Young Scientists.}



\begin{document}
	
	\newgeometry{left=30mm,right=30mm, top = 20mm}
	%\newgeometry{left=10mm,right=10mm, top = 20mm}
	\begin{abstract}
		We prove that mean decay of the coefficient of Krein system is equivalent to the mean decay of the Fourier transform of its \Szego function. 
		%We also estimate the entropy function of the Krein system in terms of the variation of the antiderivative of its coefficient.
	\end{abstract}
	
	\maketitle

	\section{Introduction}
	Let $a\in L^1_{\loc}(\R_+)$ be a complex-valued function on $\R_+ = [0,\infty)$. The Krein system with the coefficient $a$ is the following system of differential equations:
	\begin{align}\label{Krein system}
		\begin{cases}
			\frac{\partial}{\partial r}P(r,\lambda) = i\lambda P(r,\lambda) - \ol{a(r)}P_*(r,\lambda),&\quad \,\,P(0,\lambda) = 1,\\
			\frac{\partial}{\partial r}P_*(r,\lambda) = - a(r) P(r, \lambda),&\quad P_*(0,\lambda) = 1.
		\end{cases}
	\end{align}
	It was first introduced by M. Krein in \cite{Krein} and played an important role in the studies of the spectral theory of differential operators. Krein systems are often used for transferring ideas from the theory of orthogonal polynomials on the unit circle to the spectral theory of self-adjoint operators with simple spectrum. Many of the important results on the orthogonal polynomials have their counterparts in the language of Krein systems.  
	For instance, continuous versions of the Bernstein-\Szego approximations, Baxter's theorem, \Szego and strong \Szego theorems from the theory of orthogonal polynomials can be found in the survey \cite{Denisov2006} by S. Denisov among the key facts of the theory of Krein systems and spectral theory of Dirac operators, also see \cite{Denisov2003} for the continuous version of the Rakhmanov's theorem and \cite{Gubkin2021} for the ``continuous''  M\'{a}t\'{e}-Nevai-Totik theorem. 
	In the present paper we focus on another classical theorem describing probability measures with exponentially small recurrence coefficients -- the Nevai-Totik theorem \cite{Nevai1989} from 1989. The spectral version of Nevai-Totik theorem in the discrete situation (for Jacobi matrices) has been proved by D. Damanik and B. Simon in \cite{Damanik2006}. 
	The continuous setting remained open until recently. In \cite{Gub2024} we described the class of Dirac operators with exponentially decaying entropy in terms of corresponding spectral measures.  The main result of the present paper, see Theorem~\ref{thm: main equivalences of convergences} below, can be regarded as a continuous version of the Nevai-Totik theorem in the superexponentially decaying situation. To formulate it, we need to recall the definitions of some basic objects in the spectral theory of Krein systems. We will use \cite{Denisov2006} as a main reference.
	
	\medskip
	
	For any Krein system \eqref{Krein system} there exists a unique Borel measure $\sigma$ on the real line $\R$ such that $\int_{\R}(1 + x^2)^{-1}d\sigma(x) < \infty$
	and the mapping
	\begin{align}
		\label{isometry property of spectral measure}
		\mathcal{O}\colon f\mapsto \frac{1}{\sqrt{2\pi}}\int_0^{\infty}f(r)P(r,\lambda)\,dr
	\end{align}
	is a densely defined isometry between the spaces $L^2(\R_+)$ and $L^2(\R, \sigma)$. This measure is called the spectral measure of \eqref{Krein system}. If $a\in L^2(\R_+)$ then $\sigma$ belongs to the \Szego class on~$\R$. The latter means $\int_\R\frac{|\log w(x)|}{1 + x^2}\,dx < \infty$, where $w$ is the density of $\sigma$ with respect to the Lebesgue measure on $\R$. In this case the function
	\begin{align*}
		\Pi(\lambda) = \exp\left[ -\frac{1}{2\pi i}\int\limits_{-\infty}^{\infty}\left(\frac{1}{s - \lambda} - \frac{s}{s^2 + 1}\right)\log w\, ds\right],\quad\lambda\in\Cm_+
	\end{align*}
	is outer in $\Cm_+=\{\lambda\colon \Im\lambda > 0\}$, satisfies $\Pi(i) > 0$ and  $|\Pi(x)|^{-2} =w(x)$ for Lebesgue almost all $x\in \R$, see Section 4 in \cite{Garnett}. The function $\Pi$ is called  the inverse \Szego function of system~\eqref{Krein system}.
	
	Given a function $a$, one can consider  Krein systems with the coefficients $a_r\colon x\mapsto a(x + r)$ for every $r\ge 0$. Denote the corresponding  spectral measures by $\sigma_r$ and let $w_r$ be their densities with respect to the Lebesgue measure on $\R$. The entropy function of $a$ is defined~by
	\begin{gather}
		\label{def: measure entropy}
		\mathcal{K}_a(r) = \log \left(\frac{1}{\pi}\int_\R\frac{d\sigma_r(x)}{x^2 + 1}\right) - \frac{1}{\pi}\int_\R\frac{\log w_r(x)}{x^2 + 1}\,d x.
	\end{gather}
	If $\sigma$ belongs to the \Szego class then so does $\sigma_r$ for every $r\ge 0$, see \cite{Bessonov2021}. This means that $\K_a$ is well-defined (the integrals in \eqref{def: measure entropy} converge) at least for $a\in L^2(\R_+)$. It is known, see Lemma 2.3 in \cite{Bessonov2021} that $\K_a(r)\to 0$ as $r\to\infty$.
	\medskip

	\noindent\textbf{Notation.}
	We will use the notation $\ls$ and $\gs$ meaning that the corresponding inequality $\le$ or $\ge$ holds with some multiplicative constant. 
	We will use the symbol $\approx$ when both $\ls$ and $\gs$ hold. Given a function $f$ on $\R_+$ and $\alpha > 1$, we will write $f(r) = \eps_{\alpha}(r)$ if for some $c > 0$ we have $|f(r)|\ls e^{-c r^{\alpha}}$. The equality $f(r) = \eps_1(r)$ will be used when $f$ is superexponentially decaying, i.e., when for every $\delta > 0$ we have $|f(r)|\ls e^{-\delta r}$. 
	\medskip
	
	\noindent\textbf{Oscillating potentials.} For $\alpha\ge 1$ consider the following subspace $\osc_\alpha$ of $L^2(\R_+)$:
	\begin{gather*}
		\osc_{\alpha} = \left\{f\in L^2(\R_+)\colon \int_0^{\infty}f(x)\,dx\text{ converges and } \int_r^{\infty}f(x)\,dx = \eps_{\alpha}(r)\right\}.
	\end{gather*}
	The assertion $\int_r^{\infty}f(x)\,dx = \eps_{\alpha}(r)$ evidently holds when $f$ has compact support or when $f(r) = \eps_{\alpha}(r)$. It also holds for a wider class of rapidly oscillating functions of relatively weak decay, see Figure \ref{fig: oscillating potential}.
	%It is also holds for functions of the type $f(r) = \sin(e^r)/(1 + r)$ of relatively weak decay as $r\to\infty$. The oscillation of such functions causes the rapid decay of their antiderivatives, see Figure \ref{fig: oscillating potential}.
	\input{oscillating_potential_image}
	% \begin{figure}
		% \begin{tikzpicture}
			%     \draw[gray,thin] (0,-2) grid[xstep=pi/2,ystep=1] (4*pi,2); %Grid
			%     \draw   (0, 0)  --  (4*pi,0)  %X-Axis
			%     (0,-2)  --  (0 ,2);   %Y-Axis
			%     % Sine-Wave, \x r means to convert '\x' from degrees to radians
			%     \draw[blue] plot[domain=0:4*pi, samples=320] (0.5*\x,{2*sin(pow(2,\x) r)});
			% \end{tikzpicture}
		% \end{figure}
	
	\medskip

	\noindent\textbf{Functions with the decaying Fourier transform.} Let us introduce the class 
	\begin{gather}
		\label{def: decaying Fourier}
		\dfa = \left\{f\in L^2(\R)\colon \supp\left(\F f\right)\subset \R_+ \text{ and } \int_r^{\infty}|(\F f)(\xi)|^2d\xi = \eps_{\alpha}(r)\right\}
	\end{gather} of the $L^2$ functions with the decaying Fourier transform. Here $\F$ stands for the isometric on $L^2(\R)$ Fourier transform initially defined on simple functions by
	\begin{gather*}
		(\mathcal{F}f)(\xi) = \frac{1}{\sqrt{2\pi}}\int_{\R}e^{-it\xi}f(t)\,dt.
	\end{gather*}
	As we will show in Lemma \ref{Lem: salpha description}, the class $\df_{\alpha}$ consists of entire functions.
	The following theorem is the main result of the present paper.
	\bigskip
	
	\begin{Thm}\label{thm: main equivalences of convergences}
		Consider Krein system \eqref{Krein system} with the coefficient $a \in L^2(\R_+)$. For every $\alpha \ge 1$, the  following assertions are equivalent:
		\begin{center}
			\begin{enumerate*}[itemjoin={\quad\quad\quad\quad\;\,}, label=\textsc{(}\textbf{\textsc{\Alph*}}\textsc{)}]
				\item\label{cond: main thm 1} $a\in\osc_{\alpha}$\textup{;}
				\item\label{cond: main thm 2} $\sigma$ is a.\,c. and  $\frac{\Pi - \Pi(i)}{x - i}\in \df_{\alpha}$\textup{;}
				\item\label{cond: main thm C} $\K_a(r)=\eps_\alpha(r)$\textup{.}
			\end{enumerate*}
		\end{center}
		Moreover, if $\alpha > 1$ and $a\not\equiv 0$ in $L^2(\R_+)$ then the above assertions are also equivalent to 
		\begin{enumerate}[label=\textsc{(}\textbf{\textsc{\Alph*}}\textsc{)}]
			\setcounter{enumi}{3}
			\item for some $z_0\in \Cm_+$ we have  $P(r, z_0) = \eps_{\alpha}(r)$.\label{cond: main thm 3}
		\end{enumerate}
	\end{Thm}
	\noindent    Let us give some additional remarks: we can change the point $i$ in assertion \ref{cond: main thm 2} to an arbitrary $z_0\in\Cm_+$, namely, in Proposition \ref{prop: salpha does not depend on z_0} we show that \ref{cond: main thm 2} is equivalent to  
	
	\begin{minipage}{7in}
		\begin{equation}\leqnomode
			\tag{\textbf{\textsc{B'}}} \label{cond 2'}
			\hspace{-3cm}\sigma \text{ is a.\,c. and  for some } z_0\in \Cm_+ \text{ we have } \frac{\Pi - \Pi(z_0)}{x - z_0}\in \df_{\alpha}; 
		\end{equation}
	\end{minipage}\\
	when $\alpha = 1$, the implication \ref{cond: main thm 3} $\Longrightarrow$ \ref{cond: main thm 1} still holds, see Proposition \ref{prop: from small p}, however the converse may fail; points $z_0$ satisfying assertion \ref{cond: main thm 3} are exactly complex conjugate of resonances of the corresponding Dirac operator, see Section \ref{section: entropy function}.

	It is widely known that the oscillation may compensate the growth of the potential and lead to the properties typical to the properties of decreasing potentials, see \cite{Matveev1972}, \cite{Skriganov1973}, \cite{Sasaki2007} and  Appendix $2$ to XI.$8$ in \cite{ReedSimon3}. The novelty of Theorem \ref{thm: main equivalences of convergences} is implication \ref{cond: main thm 2} $\Longrightarrow$ \ref{cond: main thm 1} which allows to estimate the mean decay of the potential in terms of its spectral data; in comparison with the results from \cite{Gub2024}, Theorem \ref{thm: main equivalences of convergences} has a more explicit condition for the coefficient $a$. Description for the class of compactly supported $L^2$ potentials in terms of \Szego functions was established in the paper \cite{Korotyaev2021} by E. Korotyaev, similar result for the Schr\"{o}dinger operator is proved in \cite{Baranov2015DeBF} by A. Baranov, Y. Belov, and A. Poltoratski. Spectral properties of superexponentially decaying potentials were studied in \cite{Froese1997AsymptoticDO}, \cite{Hitrik1999BoundsOS}.
	\medskip
	
	\subsection{Structure of the paper.} In Section \ref{section: krein systems} we give the necessary background on the theory of Krein systems. Section \ref{section: proof of the main theorem} is devoted to the proof of Theorem \ref{thm: main equivalences of convergences}, in Section \ref{section: second theorem} various estimates of the entropy function are established. In the next section we discuss the Nevai-Totik theorem from the theory of orthogonal polynomials and its relation to Theorem~ \ref{thm: main equivalences of convergences}. 
	
	\subsection{Acknowledgements}
	I would like to thank Roman Bessonov for helpful discussions and comments on the manuscript.
	
	\section{Orthogonal polynomials on the unit circle}\label{section: othogonal polynomiasl}
	
	\subsection{Basics of the theory} 
	% Let us explain the motivation behind Theorem \ref{thm: main equivalences of convergences} arising from the theory of polynomials orthogonal on the unit circle.
	Let us introduce all the necessary concepts from the theory of orthogonal polynomials on the unit circle to formulate the Nevai-Totik theorem. We refer to the book \cite{Simon} by B. Simon for the general background on the theory.
	%In this subsection we give a motivation for Theorem \ref{thm: main equivalences of convergences} arising from the theory of polynomials orthogonal on the unit circle. As a main reference we use the book \cite{Simon} by Simon; the main idea of the proof below is taken from the \Szego's book \cite{Szego}.
	\medskip
	
	Let $\D=\{\omega\colon |\omega| < 1\}$ be a unit disk in the complex plane and $\T = \partial \D$ be the unit circle. Consider a probability measure $\mu$ on $\T$ which support is not a finite set, such measures are called nontrivial.
	The functions $\{z^n\}_{n\ge 0}$ are linearly independent in  $L^2(\T, \mu)$ and by the Gram-Schmidt orthogonalization procedure, we can construct the sequence $\{\Phi_n\}_{n\ge 0}$ of monic polynomials orthogonal in $L^2(\T, \mu)$. There are complex numbers $\alpha_n\in \D$ such that for $z\in \Cm$ we have
	\begin{gather}
		\label{eq: rec relation 1}
		\Phi_{n + 1}(z) = z\Phi_n(z) - \ol{\alpha_n}  \Phi_n^*(z),
		\\
		\label{eq: rec relation 2}
		\Phi_{n + 1}^*(z) = \Phi_n^*(z) - \alpha_n z \Phi_n(z),
	\end{gather}
	where $\Phi_n^*(z) = z^n\ol{\Phi_n(1/\ol{z})}$. These numbers are called the recurrence coefficients corresponding to $\mu$. The \Szego theorem states that $\sum_{n\ge 0}|\alpha_n|^2 < \infty$ if and only if $\mu$ belongs to the \Szego class on $\T$, i.e., $\displaystyle\log \mu' \in L^1(\T)$, where $\mu'$ is the density of $\mu$ with respect to the normalized Lebesgue measure $\mm$ on $\T$. In this situation there exists an outer function $\Pi$ in $\D$ such that $\Pi(0) > 0$ and ${|\Pi(\zeta)|^{-2}=\mu'(\zeta)}$ for almost every $\zeta\in \T$. The function $\Pi$ is called the inverse \Szego function of $\mu$. Theorem 2.3.5 in \cite{Simon} states that, for all $z\in \D$, $\Pi$ satisfies the limit relation
	\begin{gather}
		\label{eq: limit of phi_n^* in D}
		\lim_{n\to\infty} \Phi_n^*(z) = \Pi(z) / \Pi(0).
	\end{gather}
	%holds for every $z\in \D$, see Theorem 2.3.5 in \cite{Simon}.
	
	\subsection{Nevai-Totik theorem}  
	If $\mu = \mu'\dm$ is an a.\,c.\ measure from the \Szego class on the unit circle let $r_{\Pi}$ denote the radius of convergence of Taylor series of $\Pi$ with center at $0$. Otherwise set $r_{\Pi} = 1$. Nevai-Totik theorem, see Theorem 1 in the original paper by P. Nevai and V. Totik \cite{Nevai1989} or  Chapter 7 in \cite{Simon}, states 
	\begin{gather*}
		r_{\Pi}^{-1} = \limsup_{n\to\infty} |\alpha_n|^{1/n}.
	\end{gather*}
	%Theorem 1.5 in \cite{Gub2024} is a version of Nevai-Totik theorem in the setting of the Krein systems; Theorem \ref{thm: gubkin masters} corresponds to the case $r_{\Pi} = +\infty$,  see Section 1.2 in \cite{Gub2024} for more details.
	
	When $r_{\Pi} = +\infty$, i.e., when $\mu$ is a.\,c. from the \Szego class and $\Pi$ is entire, Nevai-Totik theorem gives  $\limsup_{n\to\infty} |\alpha_n|^{1/n} = 0$. In the next theorem we show that the order of $\Pi$ can also be calculated in terms of the sequence $\alpha_n$. Theorem \ref{thm: main equivalences of convergences} can be considered as a version of this theorem for Krein systems.
	\begin{Thm} \label{thm: MNT with order}
		The following assertions are equivalent
		\begin{enumerate}
			\item the series $\sum_{n\ge 0}\alpha_n z^n$ defines an entire function of order $\rho$;
			\item $\mu$ is a.\,c.\, measure from the \Szego class and $\Pi$ has an entire extension of order $\rho$.
		\end{enumerate}
	\end{Thm}
	\begin{proof}
		The proof is based on the relation between the order of the entire function and the asymptotic behaviour of its Taylor coefficients. Namely, let $f = \sum_{n\ge 0}f_nz^n$ be an entire function then, see Lecture 1 in \cite{Levin},  its order $\rho(f)$ can be calculated by the formula 
		\begin{gather}\label{eq: order of the entire function}
			\rho(f) = \limsup_{n\to\infty}\frac{n\ln n}{-\ln |f_n|}.
		\end{gather}
		By the Nevai-Totik theorem we already know that $\sum_{n\ge 0}\alpha_n z^n$ and $\Pi$ are entire simultaneously. Hence further we can assume that both $\sum_{n\ge 0}\alpha_n z^n$ and $\Pi=\sum_{n\ge 0}c_n z^n$ are entire and that $\mu$ is a.\,c.\,from the \Szego class. Let us show that the orders $\rho_{\alpha}$ and $\rho_{\Pi}$ of these functions are equal.
		
		First, we prove $\rho_{\Pi}\ge \rho_{\alpha}$. In the light of \eqref{eq: order of the entire function}, we need to show
		\begin{gather}\label{eq: ineq order of alpha}
			\rho = \rho_{\Pi}= \limsup_{n\to\infty} \frac{n\ln n}{-\ln |c_n|} \ge \limsup_{n\to\infty} \frac{n\ln n}{-\ln |\alpha_n|} = \rho_{\alpha}.
		\end{gather}
		If $\rho_{\Pi} = +\infty$ this inequality is trivial and below we work with the case of finite $\rho_{\Pi}$.
		Let $\mathcal{P}_n$ be the set of polynomials of degree not greater than $n$. Consider the minimization Christoffel function 
		\begin{gather}
			\label{eq: OPUC christoffel function}
			\lambda_n (z) = \lambda_n (\mu, z) = \inf\left\{\frac{\norm[P]^2_{L^2(\T, \mu)}}{|P(z)|^2}\colon P\in \mathcal{P}_n,\; P(z)\neq 0 \right\}, \quad z\in \Cm.
		\end{gather}
		For the Christoffel function we have, see Chapter 2.2 in \cite{Simon},
		\begin{gather}
			\label{lambda n formula}
			\lambda_n(0) = \prod\limits_{k = 0}^{n-1}(1 - |\alpha_k|^2),
			\qquad
			\lambda_{\infty}(0) = \inf_n\lambda_n(0)  = |\Pi(0)|^{-2}.
		\end{gather}
		Let $h_n = \sum_{k\ge 0}^{n} c_k z^k$ be the $n$-th Taylor polynomial of $\Pi$. We have $h_n\in \mathcal{P}_n$, $h_n(0) = \Pi(0)$ and $d\mu(\zeta) = |\Pi(\zeta)|^{-2}\dm$ hence
		\begin{align*}
			\lambda_n(0) &\le \frac{\norm[h_n]^2_{L^2(\T, \mu)}}{|h_n(0)|^2}
			= \frac{1}{|h_n(0)|^2}\int_{\T} \left|h_n(\zeta)\right|^2d\mu(\zeta)
			\\
			&= \frac{1}{|\Pi(0)|^2}\int_{\T} \left|h_n(\zeta)\Pi^{-1}(\zeta)\right|^2\dm(\zeta)  = \frac{1}{|\Pi(0)|^2}\int_{\T} \left|1 + \frac{h_n(\zeta) - \Pi(\zeta)}{\Pi(\zeta)}\right|^2\dm(\zeta) 
			\\
			&=\frac{1}{|\Pi(0)|^2}\int_{\T} 1 + 2\Re\left(\frac{h_n(\zeta) - \Pi(\zeta)}{\Pi(\zeta)}\right) + \left|\frac{h_n(\zeta) - \Pi(\zeta)}{\Pi(\zeta)}\right|^2\dm(\zeta).
		\end{align*}
		The function $\displaystyle \frac{h_n - \Pi}{\Pi}$ is analytic in $\D$, therefore the second term vanishes after the integration.  This implies
		\begin{align}\label{eq: lambda tmp}
			\lambda_n(0)\le \frac{1}{|\Pi(0)|^2}\int_{\T} 1 &+  \left|\frac{h_n(\zeta) - \Pi(\zeta)}{\Pi(\zeta)}\right|^2\dm(\zeta) = \lambda_{\infty}(0)  + \int_{\T}  \left|\frac{h_n(\zeta) - \Pi(\zeta)}{\Pi(\zeta)}\right|^2\dm(\zeta).
			%+ O\left(e^{-\frac{2n\ln n }{\rho + \eps}}\right),\quad n\to\infty. 
		\end{align}
		For $\zeta\in \T$ we can write the uniform bound $\left|\Pi(\zeta) - h_n(\zeta)\right| \le \sum_{m > n}|c_m|$. Formula \eqref{eq: order of the entire function} for $\Pi$ implies that for every $\eps > 0$ and large $n$ the inequality $\rho + \eps \ge n\ln n / (-\ln |c_n|)$ holds. This is equivalent to $c_n \le \exp\left({-\frac{n\ln n }{\rho + \eps}}\right)$. Therefore we have  
		\begin{gather*}
			\left|\Pi(\zeta) - h_n(\zeta)\right| \le \sum_{m > n}|c_m| \le \sum_{m > n} e^{-\frac{m\ln m }{\rho + \eps}} \ls e^{-\frac{n\ln n }{\rho + \eps}}.
		\end{gather*}
		Moreover, $\Pi$ is separated from $0$ on $\T$, otherwise the assertion $\mu(\T) < \infty$ would fail. This means that the integral in \eqref{eq: lambda tmp} is $O\left(e^{-\frac{2n\ln n }{\rho + \eps}}\right)$ as $n\to\infty$. Then the relations in \eqref{lambda n formula} give
		\begin{gather*}
			\prod\limits_{k = 0}^{n-1}(1 - |\alpha_k|^2) \le \prod\limits_{k = 0}^{\infty}(1 - |\alpha_k|^2) + O\left(e^{-\frac{2n\ln n }{\rho + \eps}}\right),\quad n\to\infty.
		\end{gather*}
		Therefore $|\alpha_n| =O\left(e^{-\frac{n\ln n }{\rho + \eps}}\right) $ as $n\to\infty$ and \eqref{eq: ineq order of alpha} follows. This proves the inequality $\rho_{\Pi}\ge \rho_{\alpha}$
		%the order of the function $\sum \alpha_n z^n$ is not greater than  $\rho + \eps$ by the same formula of order in terms of the coefficients.
		\medskip
		
		The proof of the $\rho_{\alpha}\ge \rho_{\Pi}$ is simpler and uses the same argument as in the proof of Theorem 1.1 from \cite{Simon2006}.
		% Assume that the series $\sum_{n\ge 0} \alpha_n z^n$ converges everywhere and has order $\rho$. By the Nevai-Totik theorem we already know that $\mu$ is a.\,c.\,and $\Pi$ is entire. Thus we only need to bound the order of $\Pi$. 
		For $\zeta \in \T$ we have 
		% From the recurrence relation \eqref{eq: rec relation 1} and 
		$|\Phi_n^*(\zeta)| = |\zeta^n\ol{\Phi_n(1/\ol{\zeta})}| = |\Phi_n(\zeta)|$ hence \eqref{eq: rec relation 1} implies
		\begin{gather*}
			|\Phi_{n + 1}(\zeta)| \le |\zeta\Phi_n(\zeta)| + |\ol{\alpha_n}  \Phi_n^*(\zeta)| = (1 + |\alpha_n|)|\Phi_n(\zeta)|.
		\end{gather*}
		Inductively we deduce  
		\begin{gather*}
			|\Phi_n^*(\zeta)| = |\Phi_n(\zeta)| \le \prod_{k = 0}^{n-1} (1 + |\alpha_k|) < \infty, \quad \zeta\in \T.
		\end{gather*}
		Therefore $z^{-n}\Phi_n$ is bounded on $\T$ uniformly in $n$. All $\Phi_n$ are monic hence $z^{-n}\Phi_n(z) = 1 + o(1)$ as $|z|\to\infty$. Now maximum modulus principle implies boundedness of $z^{-n}\Phi_n$ in the domain $\Cm\setminus\D$. Therefore, by \eqref{eq: rec relation 2} we get
		\begin{gather*}
			\sum_{n = 0}^{\infty}|\Phi_{n + 1}^*(z) - \Phi_n^*(z)| = \sum_{n = 0}^{\infty}|z\alpha_n\Phi_n(z)| \ls \sum_{n = 0}^{\infty}|z^{n + 1}\alpha_n|.
		\end{gather*}
		In particular, this means that $\Phi_n^*$ converge on the compact subsets of $\Cm$ and \eqref{eq: limit of phi_n^* in D} holds for every $z\in\Cm$.
		%This limit coincides with $\Pi/\Pi(0)$ in $\D$, recall \eqref{eq: limit of phi_n^* in D}. Hence $\Pi$ has an analytic continuation in $\Cm$ and satisfies 
		Moreover, this gives the estimate $|\Pi(z)|\ls \sum_{n = 0}^{\infty}|z^{n + 1}\alpha_n|$. The inequality $\rho_{\Pi}\le \rho_{\alpha}$ follows and the proof is concluded.
	\end{proof}
	
	\section{Krein systems}\label{section: krein systems}
	Consider the Krein system \eqref{Krein system}, let $\sigma$ be its spectral measure and $\Pi$ be the corresponding inverse \Szego function. A simple calculation shows that for $r\ge 0$ and $z\in \Cm$ we have
	\begin{gather}
		\label{Krein system reflection formula}
		P(r, z) = e^{i z r}\ol{P_*(r,\ol{ z})},\qquad P_*(r, z) = e^{i z r}\ol{P(r,\ol{ z})}.
	\end{gather}
	Furthermore, for every $\lambda,\mu\in\Cm$, the functions $P$, $P_*$ satisfy the Christoffel-Darboux formula
	\begin{gather}
		\label{first CD formula}
		P(r,\lambda)\ol{P(r,\mu)} - P_*(r,\lambda)\ol{P_*(r,\mu)} = i(\lambda - \ol{\mu})\int_0^r P(s,\lambda)\ol{P(s,\mu)}\,ds,
	\end{gather}
	which is also proved by a straightforward calculation, see Lemma 3.6 in \cite{Denisov2006}. If we let $\mu = \lambda$ then this becomes
	\begin{gather}
		\label{second CD formula}
		|P_*(r,\lambda)|^2 - |P(r,\lambda)|^2 = 2\Im \lambda \int_0^r |P(s,\lambda)|^2\, ds.
	\end{gather}
	Krein theorem, see Section 8 in \cite{Denisov2006} or Section 3 in \cite{Teplyaev2005}, states that $\sigma$ belongs to the \Szego class on the real line if and only if for every $\lambda_0\in \Cm_+$ we have $P(\cdot, \lambda_0)\in L^2(\R_+)$. In this situation there exists a constant $\gamma \in [0, 2\pi)$ and a sequence of positive numbers $r_n\to\infty$ such that the limit relation
	\begin{gather}
		%\label{eq: convergence in Krein theorem}
		\label{eq: tmp convergence in Krein theorem}
		\lim_{n\to\infty}P_*(r_n,\lambda) = e^{i\gamma}\Pi(\lambda) = \Pi_{\gamma}(\lambda)
	\end{gather}
	holds for every $\lambda\in \Cm_+$. Convergence $\lim P_*(r, \lambda) = \Pi_{\gamma}(\lambda)$ as $r\to\infty$ takes place when $a\in L^2(\R_+)$, see Lemma \ref{Lem: convergence in Krein theorem} below.
	Equations \eqref{second CD formula} and \eqref{eq: tmp convergence in Krein theorem} together imply
	\begin{gather}
		\label{expression for the szego function}
		|\Pi(\lambda)|^2 = 2\Im \lambda \int_0^{\infty} |P(s,\lambda)|^2\, ds,\quad \lambda\in\Cm_+. 
	\end{gather}
	Theorem 6.2 in \cite{Denisov2006} states that $|P_*(r,x)|^{-2}\,dx\to d\sigma(x)$ in the weak - $\ast$ sense. As a corollary of this convergence we get the following important lemma.
	\begin{Lem}\label{Lem: convergence implies a.c.}
		If $|P_*(r, x)|\to |\Pi(x)|$ uniformly on compact subsets of $\R$ then $\sigma$ is absolutely continuous. 
	\end{Lem}
	\subsection{Extremal problem and Christoffel functions}
	
	Let $PW_{[0, r]}$ denote the Paley-Wiener space of entire functions $f$ with the spectrum in $[0,r]$, in other words, the space of functions of the form $f = \F^{-1}\phi$ with $\phi\in L^2([0,r])$.
	Lemma 8.1 in \cite{Denisov2006} states that $PW_{[0,r]}\subset L^2(\R, \sigma)$.
	For $r >0$ and $z\in \Cm$, define
	\begin{gather}\label{definition of the minimization function for Krein system}
		\bm_r(z) = \inf\left\{\frac{1}{2\pi|f(z)|}\norm[f]^2_{ L^2(\R, \sigma)} \colon  f\in PW_{[0,r]},\; f(z) \neq 0\right\}.
	\end{gather}
	The function $\bm_r$ is the analog of the Christoffel function $\lambda_n$ from the theory of orthogonal polynomials, recall \eqref{eq: OPUC christoffel function}. Lemma 8.2 in \cite{Denisov2006} says
	\begin{gather}
		\label{m_r function formula}
		\bm_r(z) = \left(\int_0^r|P(s, z)|^2\,ds\right)^{-1},\quad z\in \Cm.
	\end{gather}
	Moreover, the minimizer in \eqref{definition of the minimization function for Krein system} is unique up to the constant factor and is given by 
	\begin{gather}
		\label{eq: reproducing kernel in PW}
		k_{r,z}(\lambda) = \frac{1}{2\pi}\int_0^r P(s,\lambda)\ol{P(s,z)}\, ds\in PW_{[0,r]}.
	\end{gather}

	\subsection{Krein system with \texorpdfstring{$L^2$}{L2} coefficient}
	In the present paper we are interested in the case when the coefficient $a$ of the Krein system \eqref{Krein system} belongs to $L^2(\R_+)$. Three following results describe the properties of Krein system in this situation.
	
	\begin{Thm}[S. Denisov, \cite{Denisov2006}, Theorem 11.2]\label{thm: szego to hardy}
		If $a\in L^2(\R_+)$ then $\sigma$ belongs to the \Szego class on the real line, $\Pi$ is well-defined and $\Pi_{\gamma}^{-1} = 1 + h$, where $h\in H^2(\Cm_+)$ is such that $\|h\|_{H^2(\Cm_+)} = \|a\|_{L^2(\R_+)}$.
	\end{Thm}
	\begin{Prop}\label{Prop: Krein system estimates}
		Assume that $a\in L^2(\R_+)$. Then, for every $\eps > 0$, the function $P_*(r, z)$ is uniformly bounded for $r\ge 0$ and $z$ with $\Im z > \eps$. Also for $z\in \Cm$ we have 
		\begin{gather*}
			|P_*(r, z)|\le \exp\left(\|a\|_{L^1([0,r])} + r(\Im z)_-\right) \ls \exp\left(r\|a\|_{L^2(\R_+)} + r(\Im z)_-\right),
		\end{gather*}
		where $(x)_-$ is the negative part of $x$, i.e.,  $(x)_- = 0$ if $x\ge 0$ and $(x)_- = -x$ if $x < 0$.
	\end{Prop}
	\begin{proof}
		The boundedness of $P_*$ follows from  \GrBell inequality applied for the Krein system, for the details see the proof of Theorem 11.1 in  \cite{Denisov2006}. Different application of \GrBell inequality gives the bound 
		\begin{gather*}
			|P_*(r,z)|\le \exp(\|a\|_{L^1([0,r])}) = \exp\left(\|a\|_{L^1([0,r])} + r(\Im z)_-\right)
		\end{gather*}
		for $z$ with $\Im z \ge 0$, see the proof of Theorem 12.1 in  \cite{Denisov2006}. For $z$ with negative imaginary part we can use the reflection  formula \eqref{Krein system reflection formula} and the inequality $|P_*(r, z)|> |P(r,z)|$ for $z\in \Cm_+$ given by \eqref{first CD formula}:
		\begin{gather*}
			|P_*(r, z)| = \left|e^{i z r}\ol{P(r,\ol{ z})}\right| \le e^{r(\Im z)_-}|P_*(r,\ol{ z})|\le \exp\left(\|a\|_{L^1([0,r])} + r(\Im z)_-\right).
		\end{gather*}
		The inequality $ \|a\|_{L^1([0,r])}\le \sqrt{r}\|a\|_{L^2([0,r])}\le \frac{1 + r}{2}\|a\|_{L^2([0,r])}$ finishes the proof.
	\end{proof}
	\begin{Lem}
		\label{Lem: convergence in Krein theorem}
		If $a\in L^2(\R_+)$ then for some $\gamma\in [0, 2\pi)$ we have $\lim_{r\to\infty}P_*(r,\lambda) = \Pi(\lambda)$ for every $\lambda\in \Cm_+$. 
	\end{Lem}
	\begin{proof}
		Apply the Cauchy inequality to the differential equation for $P_*$ and use the assertion $\|P\|_{L^2(\R_+)}<\infty$ from Krein theorem. We have
		\begin{gather*}
			\| \frac{\partial}{\partial r}P_*(r,\lambda)  \|_{L^1(\R_+)}\le \|a\|_{L^2(\R_+)}\|P(r,\lambda)\|_{L^2(\R_+)} < \infty
		\end{gather*}
		hence $P_*(r,\lambda)$ converges as $r\to\infty$. The limit coincides with $\Pi_{\gamma}$, recall \eqref{eq: tmp convergence in Krein theorem}.
		
	\end{proof}
	
	\subsection{Entropy function}\label{section: entropy function}
	Consider Krein system \eqref{Krein system} with coefficient $a\in L^2(\R_+)$ and let $J = \smatrix{0}{1}{-1}{0}$ be the square root of the minus identity matrix and $Q =\smatrix{-q}{p}{p}{q}$ be the matrix-valued function with $p(r) = -2\Re a(2r)$, $q(r)  = 2\Im a(2r)$. Krein system with the coefficient $a$ is equivalent to the differential equation for the generalized eigenfunctions of the Dirac operator on the half-line
	\begin{gather*}
		\mathcal{D}_Q = J\frac{d}{dr} + Q,
	\end{gather*}
	see Section 13 in \cite{Denisov2006} for the details. In particular, the spectral measure $\sigma_a$ can be defined in terms of $D_Q$ and the results for Krein systems, such as Theorem \ref{thm: main equivalences of convergences}, can be reformulated for the Dirac operator. When the inverse \Szego function of the Krein system or Dirac operator is entire one can speak of its zeroes -- the scattering resonances, see the book \cite{Dyatlov2019} by S. Dyatlov and M. Zworski for the general theory. The exposition for specific case of the Dirac operator can be found in \cite{Korotyaev2021}, also see the references within. Theorem \ref{thm: main equivalences of convergences} allows us to study resonances of Dirac operators with oscillating potentials from $\osc_\alpha$. Such studies will be presented elsewhere.

	In the papers \cite{BESSONOV2020106851},  \cite{Bessonov2021} R. Bessonov and S. Denisov described the class of canonical systems with the spectral measure from the \Szego class in terms of the so-called entropy function, also see \cite{Bessonov2022} for the case of Dirac operators. Let us formulate this result on the language of Krein systems. Let $N_a$ be the solution of 
	\begin{gather}
		\label{eq: Na def}
		JN_a'(r) + Q(t)N_a(r) = 0,\quad N_a(0) =  \left(\begin{smallmatrix}1&0\\0&1\end{smallmatrix}\right),\quad r\ge 0
	\end{gather}
	and set
	\begin{gather}
		\label{eq: entropy def}
		\qquad E_a(r) = \det\left[\int_{r}^{r + 2}  N_a^*(t)N_a(t)\,dt\right] - 4.
	\end{gather}
	Recall the definition \eqref{def: measure entropy} of the entropy function $\K_a$. We have $\K_a(0) < \infty$ if and only if $\sigma_a$ belongs to the \Szego class. 
	\begin{Thm}[Theorem 1.2, \cite{Bessonov2021}]\label{known theorem on two entropies}\label{thm: bessonov, denisov, entropy}
		Assume that $a\in L^1_{\loc}(\R_+)$ and let $\sigma$ be the spectral measure of the corresponding Krein system.
		Then $\sigma$ belongs to the \Szego class on the real line if and only if $\sum_{n\ge 0}E_a(n) <\infty$. More precisely, we have
		\begin{gather*}
			\K_a(0)\ls \sum_{n\ge 0}E_a(n)\ls \K_a(0)^{c\K_a(0)}
		\end{gather*}
		for some absolute constant $c$.
	\end{Thm}

	The paper \cite{Gub2024} of the author is dedicated to the case when the series $\sum_{n\ge 0}E_a(n)$ converges exponentially fast. When $E_a(r) = \eps_{1}(r)$, Theorem 1.5 in \cite{Gub2024} takes the following form.
	\begin{Thm}[Theorem 1.5, \cite{Gub2024}]\label{thm: gubkin masters}
		Assume that $a\in L^2(\R_+)$ then $E_a(r) = \eps_{1}(r)$ if and only if the spectral measure $\sigma$ is a\,.c.\,and $\Pi$ has an entire extension satisfying $\Pi(x - i\delta)/(x + i)\in H^2(\Cm_+)$ for every $\delta > 0$.    
	\end{Thm}
	This theorem concerns an $\alpha = 1$ part of Theorem \ref{thm: main equivalences of convergences}. In Proposition \ref{Prop: S_1 class description} we show that the assertion $\Pi(x - i\delta)/(x + i)\in H^2(\Cm_+)$ is exactly the assertion \ref{cond: main thm 2} from Theorem \ref{thm: main equivalences of convergences}. Thus, the proof of Theorem \ref{thm: main equivalences of convergences} for $\alpha = 1$ requires the equivalence of $E_a(r) = \eps_{1}(r)$ and $a\in \osc_1$. We formulate this in the following two results. Let $g_{a,r}(t) = \int_r^ta(s)\,ds$ and define the variation of $a$ by
	\begin{gather}\label{eq: def variation of a}
		D_a(r) = 2\int_r^{r + 2}|g_{a,r}(t)|^2\,dt - \left|\int_r^{r + 2}g_{a,r}(t)\,dt\right|^2.
	\end{gather}
	\begin{Thm}\label{thm: entropy and variation}
		If $a\in L^2(\R_+)$ then  $E_a(r)\approx D_a(r)$.
	\end{Thm}
	% The variation also can be used to describe the class $\osc_\alpha$.
	\begin{Thm}\label{thm: osc var}
		If $a\in L^2(\R_+)$ then  $a\in\osc_{\alpha}$ if and only if $D_a(r) = \ea(r)$.
	\end{Thm}
	R. Bessonov and S. Denisov, see \cite{Bessonov20221}, established the connection between the entropy function and the Sobolev norm of the coefficient.
	\begin{Thm}[Theorem 4.1, \cite{Bessonov20221}] Assume that $a\in L^2(\R_+)$ then
		\begin{gather*}
			\sum_{n\ge 0}E_a(n)\approx \|a\|_{H^{-1}(\R)}^2 = \int_{\R}\frac{|(\F a)(\xi)|^2}{1 + \xi^2}\,d\xi,
		\end{gather*}
		where the quantity in the right-hand side is the definition of the norm in Sobolev space $H^{-1}(\R)$ and the constant in $\approx$ depends on the $\|a\|_{L^2(\R_+)}$.
	\end{Thm}
	Theorem \ref{thm: entropy and variation} can be derived from the results in  \cite{Bessonov20221} but we give an independent proof. The proofs of Theorems \ref{thm: entropy and variation} and \ref{thm: osc var} are mostly technical, we postpone them in the end of the present paper, Section \ref{section: second theorem}.
	
	\section{Proof of Theorem \ref{thm: main equivalences of convergences}}\label{section: proof of the main theorem}
	% \subsection{Independence of the point \texorpdfstring{$z_0$}{z0}}\label{section: independecne on the point}
	% Let us show that the  assertions \ref{cond: main thm 1} and \ref{cond: main thm 2} of Theorem \ref{thm: main equivalences of convergences} do not depend on the point $z_0$: if one of these assertions holds for some $z_0$ then it holds for every $z_0$. The following lemma concerns assertion \ref{cond: main thm 1}.
	
	\subsection{Equivalence of \texorpdfstring{\ref{cond: main thm 2}}{(B)} and \texorpdfstring{\eqref{cond 2'}}{(B')}}
	To deal with the assertion \ref{cond: main thm 2} we need to examine the properties of the class $\df_{\alpha}$. From the definition \eqref{def: decaying Fourier} we see that $\df_{\alpha}\subset \df_{\beta}$ for $\alpha \ge \beta$. In particular, $\df_1$ is the largest class. The following lemma will help us in showing that $\dfa$ consists of entire functions.
	\begin{Lem}\label{Lem: order of a sum bound}
		Let $f\colon \R_+\to \R$ be a measurable function satisfying $f(r) = \ea(r)$ with $\alpha > 1$. Let $g(x) = \sum_{n\ge 0}f(n)e^{xn}$ then $g(x)$ is bounded for $x\le 0$ and there exists a constant $c\in \R$ such that $ |g(x)|\ls \exp\left(c|x|^{\alpha^*}\right)$, where $\alpha^* = \frac{\alpha}{\alpha - 1}$.
	\end{Lem}
	\begin{proof}
		When $x\le 0$ we have $|g(x)|\le \sum_{n\ge 0}|f(n)| < \infty$ because $f(n) = \ea(n)$. Take $x > 0$. From the definition of $\ea$ we know that $f(n)\ls \exp(-n^{\alpha}/c_1)$ for some constant $c_1$. Hence we have
		\begin{gather*}
			\sum_{n\ge 0}f(n)e^{xn}\ls\sum_{n\ge 0}\exp\left(-n^{\alpha}/c_1 + xn\right).
		\end{gather*}
		Let $N_0 = \left[\left(c_1(x + 1)\right)^{1/(\alpha - 1)}\right] + 1$. Then for $n > N_0$ we have $-n^{\alpha}/c_1 + xn < -n$ and 
		\begin{gather*}
			\sum_{n\ge N_0}f(n)e^{xn}\ls\sum_{n\ge N_0} e^{-n}\ls 1.
		\end{gather*}
		On the other hand, if $n\le N_0$ then $-n^{\alpha}/c_1 + xn < xN_0$ hence
		\begin{gather*}
			\sum_{n < N_0}f(n)e^{xn}\ls \sum_{n < N_0}e^{N_0x}\le N_0e^{N_0x}.
		\end{gather*}
		The bound $N_0 = O\left(x^{1/(\alpha - 1)}\right)$ as $x\to\infty$ finishes the proof.
	\end{proof}
	\begin{Lem}\label{Lem: salpha description}
		Assume that $f\in \dfa$ with some $\alpha \ge 1$. Then $f$ has an entire continuation of order not greater than $\alpha^* = \frac{\alpha}{\alpha - 1}$. Furthermore, $f$ is bounded in every horizontal upper half-plane $\Omega_{\delta} = \{z\colon \Im z > -\delta\}$. 
	\end{Lem}
	\begin{proof}
		let $\phi = \F f$. We know that $\supp \phi\subset{\R_+}$ and $\int_{r}^{\infty}|\phi(t)|^2\,dt = \ea(r)$ hence the integral $\frac{1}{\sqrt{2\pi}}\int_0^{\infty}\phi(t)e^{izt}\,dt$ converges for every $z\in \Cm$ and defines an entire function. This entire function coincides with $f$ on $\R$ hence $f$ is entire. Also we can write
		\begin{gather*}
			\left|\int_0^{\infty}\phi(t)e^{izt}\,dt\right|\le \int_0^{\infty}|\phi(t)|e^{-t\Im z}\,dt\le \sum_{n\ge 0}\sqrt{\int_{n}^{n + 1}e^{-t\Im z}\,dt}\sqrt{\int_n^{n + 1}|\phi(t)|^2\,dt}.
		\end{gather*}
		We have $\sqrt{\int_n^{n + 1}|\phi(t)|^2\,dt} = \eps_{\alpha}(n)$ and $\int_{n}^{n + 1}e^{-t\Im z}\,dt \approx e^{-n\Im z}$ hence the estimate of the order and the required boundedness follow from  Lemma \ref{Lem: order of a sum bound}.
	\end{proof}
	In other words, $\dfa$ for $\alpha \ge 1$ consists of entire functions of order not greater than $\alpha^*$. We can formulate a different description of the class $\df_1$.
	\begin{Prop}\label{Prop: S_1 class description}
		Let $f$ be an entire function, then $f\in \df_1$ if and only if $f\in H^2(\Omega_{\delta})$ for every upper horizontal half-plane $\Omega_{\delta} = \{z\colon \Im z > -\delta\}$.
	\end{Prop}
	\begin{proof}
		Assume that $f$ belongs to the Hardy space in $\Omega_\delta$ for every $\delta > 0$.  Let $\phi$ be the Fourier transform of $f$ and $\phi_{\delta}$ be the Fourier transform of $f(x - i\delta)$. Then for every $\delta > 0$ we have $\phi_{\delta}\in L^2(\R_+)$ and $\phi_{\delta} = e^{\delta x}\phi$. Therefore the integral $\int_{\R_+}e^{2\delta x}|\phi^2(x)|\,dx$ converges for every $\delta > 0$, which is equivalent to $\int_r^{\infty }|\phi(x)|^2\,dx = \eps_1(r)$.
		
		If $\int_r^{\infty }|\phi(x)|^2\,dx = \eps_1(r)$ then $f(z) = \frac{1}{\sqrt{2\pi}}\int_0^{\infty}\phi(r)e^{irz}\, dr$, where the integral is absolutely convergent. This means $\|f\|_{H^2(\Omega_{\delta})} = \|f(x - i\delta)\|_{L^2(\R)} =\|\phi e^{\delta x}\|_{L^2(\R_+)} < \infty.$
	\end{proof}
	In the light of Proposition \ref{Prop: S_1 class description} we can reformulate Theorem \ref{thm: gubkin masters} in the following way.
	\begin{Thm}\label{thm: reformulation master}
		Assume that $a\in L^2(\R_+)$ then $E_a(r) = \eps_{1}(r)$ if and only if $\sigma$ is a.\,c.\,and for some $z_0\in \Cm_+$ we have $ \frac{\Pi - \Pi(z_0)}{z - z_0}\in \df_{1}$.
	\end{Thm}
	The description of the $\df_1$ class given in Proposition \ref{Prop: S_1 class description} implies that the assertions \ref{cond: main thm 2} and \eqref{cond 2'} of Theorem \ref{thm: main equivalences of convergences} are equivalent. In the following proposition we prove that the same is true for every $\alpha\ge 1$.
	\begin{Prop}\label{prop: salpha does not depend on z_0}
		Let $f$ be an entire function and $\alpha \ge 1$. If the assertion $\frac{f - f(z_0)}{z - z_0}\in \df_{\alpha}$ holds for some $z_0\in \Cm$ then it holds for every $z_0\in \Cm$.
	\end{Prop}
	\begin{proof}
		We have $\dfa\subset\df_1$ hence $\frac{f - f(z_0)}{z - z_0}\in\df_1$. The characterization of $\df_1$ from Proposition \ref{Prop: S_1 class description} implies $\frac{f - f(z_1)}{z - z_1}\in\df_1$ for every $z_1\in \Cm$. Let $\phi = \F\left(\frac{f - f(z_0)}{z - z_0}\right)$ and $\psi = \F\left(\frac{f - f(z_1)}{z - z_1}\right)$. We have $\phi, \psi\in L^2(\R_+)$ and 
		\begin{gather}\label{eq: F prop 2.4}
			f(z) = f(z_1) + \frac{z - z_1}{\sqrt{2\pi}}\int_{0}^{\infty}\phi(x)e^{izx}\,dx = f(z_2) + \frac{z - z_2}{\sqrt{2\pi}}\int_{0}^{\infty}\psi(x)e^{izx}\,dx,
		\end{gather}
		where the integrals are absolutely convergent for every $z\in \Cm$. Consider the functions
		\begin{gather*}
			\Phi(t) = \int_t^{\infty}\phi(x)e^{iz_2 x}\, dx, \qquad \Psi(t) = \int_t^{\infty}\psi(x)e^{iz_1 x}\, dx.
		\end{gather*}
		The proposition will follow from the equality 
		\begin{gather}\label{eq: prop 2.4 main}
			\psi(x) = \phi(x)- i(z_1 - z_2)\Phi(x)e^{-iz_2x}.
		\end{gather}
		Indeed, the assertion $\int_r^{\infty}|\phi(t)|^2\,dt=\ea(r)$ implies $\Phi(t) = \eps_{\alpha}(t)$ by the integration by parts and the required $\int_r^{\infty}|\psi(t)|^2\,dt=\ea(r)$ then follows from \eqref{eq: prop 2.4 main}.
		\medskip
		
		Let us focus on \eqref{eq: prop 2.4 main}. From \eqref{eq: F prop 2.4} we get
		\begin{gather}
			\nonumber
			f(z_2) = f(z_1) + \frac{z_2 - z_1}{\sqrt{2\pi}}\Phi(0),\qquad f(z_1) = f(z_2) + \frac{z_1 - z_2}{\sqrt{2\pi}}\Psi(0),
			\\
			\label{eq: psi(0) = phi(0)}
			\Psi(0)= \Phi(0) = \sqrt{2\pi}\frac{f(z_2) - f(z_1)}{ z_2 - z_1}.
		\end{gather}
		For $z\in \Cm$ we have
		\begin{align*}
			\int_{0}^{\infty}\phi(x)e^{izx}\,dx &= \int_{0}^{\infty}\phi(x)e^{iz_2x}\cdot e^{i(z - z_2)x}\,dx 
			\\
			&= -\Phi(x)e^{i(z - z_2)x}\Big|_0^{\infty} + i(z - z_2)\int_{0}^{\infty}\Phi(x)\cdot e^{i(z - z_2)x}\, dx
			\\
			&=
			\Phi(0) + i(z - z_2)\int_{0}^{\infty}\Phi(x)\cdot e^{i(z - z_2)x}\, dx.
		\end{align*}
		Similar transformation of  $\int_{0}^{\infty}\psi(x)e^{izx}\,dx$ in \eqref{eq: F prop 2.4} gives
		\begin{align*}
			f(z_1) &+ \frac{\Phi(0)(z - z_1)}{\sqrt{2\pi}} + \frac{i(z - z_1)(z - z_2)}{\sqrt{2\pi}}\int_{0}^{\infty}\Phi(x)\cdot e^{i(z - z_2)x}\, dx
			\\
			&=f(z_2) + \frac{\Psi(0)(z - z_2)}{\sqrt{2\pi}} + \frac{i(z - z_1)(z - z_2)}{\sqrt{2\pi}}\int_{0}^{\infty}\Psi(x)\cdot e^{i(z - z_1)x}\, dx.
		\end{align*}
		Regrouping the terms, we get
		\begin{align*}
			\sqrt{2\pi}(f(z_1) - f(z_2)) &+  z(\Phi(0) - \Psi(0)) + (\Psi(0)z_2 - \Phi(0)z_1)
			\\
			&=i(z - z_1)(z - z_2)\int_{0}^{\infty}\left(\Phi(x)e^{- iz_2x}-\Psi(x)e^{- iz_1x}\right) e^{izx}\, dx.
		\end{align*}
		The left-hand side vanishes because of \eqref{eq: psi(0) = phi(0)}. Therefore we get $\Phi(x)e^{-iz_2x} - \Psi(x)e^{-iz_1x}= 0$ or $\Psi(x) = \Phi(x)e^{i(z_1 - z_2)x}$. By the definition of $\Phi$ and $\Psi$ we have $\Phi'(x) = -\phi(x)e^{iz_2x}$ and $\Psi'(x) = -\psi(x)e^{iz_1x}$. Taking the derivative in the previous equality, we get
		\begin{gather*}
			-\psi(x)e^{iz_1 x}  = -\phi(x)e^{iz_2x}\cdot e^{i(z_1 - z_2)x} + \Phi(x)\cdot i(z_1 - z_2)e^{i(z_1 - z_2)x},
		\end{gather*}
		which is equivalent to \eqref{eq: prop 2.4 main}. The proof is finished.
	\end{proof}
	\subsection{Assertion \ref{cond: main thm 3}. Decaying solution of Krein system}
	In this subsection we prove that the assertion \ref{cond: main thm 3} of Theorem \ref{thm: main equivalences of convergences} implies assertions \ref{cond: main thm 1} and \ref{cond: main thm 2} and besides that gives other important information about $\Pi$. The following Lemma will be useful to us,
	% \begin{Lem}\label{lemma: osc prop}
		% 		If $f(t)\in \osc_{\alpha}$ then $f(t)e^{izt}\in \osc_{\alpha}$ for every $z\in \Cm_+$.
		% 	\end{Lem}
	% 	\begin{proof}
		% 		The proof is a straightforward integration by parts. Define $F(r) = \int_r^\infty f(x)\, dx$. Since $f\in \osc_{\alpha}$, $F$ is well-defined and $F(r) = \eps_{\alpha}(r)$. We have
		% 		\begin{gather*}
			% 			\int_t^{\infty}f(s)e^{izs}\, ds =  -F(s)e^{izs}\Big|_{t}^{\infty} + iz\int_t^{\infty}F(s)e^{izs}\, ds = F(t)e^{izt} + iz\int_t^{\infty}F(s)e^{izs}\, ds.
			% 		\end{gather*}
		% 		Both terms in the right-hand side of the equality are $\eps_{\alpha}(t)$. Therefore $f(t)e^{izt}\in \osc_{\alpha}$.
		% 	\end{proof}
	\begin{Prop}\label{prop: from small p}
		Assume that $a\in L^2(\R_+)$, $\alpha \ge 1$ and $z_0\in \Cm_+$ are such that $P(r,z_0)= \eps_{\alpha}(r)$. Then $\sigma$ is a.\,c., $a\in \osc_\alpha$ and $\Pi$ has an analytic continuation into the whole complex plane such that $\Pi(\ol{z_0}) = 0$ and $\frac{\Pi(z)}{z - \ol{z_0}}\in \dfa$. 
	\end{Prop}
	\begin{proof}
		Substitute $z_0$ for $\mu$ into the Christoffel-Darboux formula \eqref{first CD formula}:
		\begin{gather}
			\label{eq: rewritten C-D}
			\int_0^r P(s,\lambda)\ol{P(s,z_0)}\,ds = i\frac{P_*(r,\lambda)\ol{P_*(r,z_0)} - P(r,\lambda)\ol{P(r,z_0)}}{\lambda - \ol{z_0}}.
		\end{gather}
		We know that $|P(s,\lambda)|$ is bounded by some exponential function in $s$ by Proposition \ref{Prop: Krein system estimates} and $P(s, z_0) = \eps_{\alpha}(s)$,  hence the integral 
		\begin{gather}
			\label{eq: definition of F}
			F(\lambda) = \int_0^{\infty} P(s,\lambda)\ol{P(s,z_0)}\,ds
		\end{gather}
		converges absolutely for every $\lambda\in \Cm$ and defines an entire function. In particular, $P(r,\lambda)\ol{P(r,z_0)}\to 0$ as $r\to\infty$. If $\lambda\in \Cm_+$ then from Lemma \ref{Lem: convergence in Krein theorem} we have $P_*(r,\lambda)\ol{P_*(r,z_0)} \to \Pi_{\gamma}(\lambda)\ol{\Pi_{\gamma}(z_0)}$ as $r\to\infty$ hence the right-hand side of \eqref{eq: rewritten C-D} converges as $r\to\infty$ and 
		\begin{gather*}
			F(\lambda) = \frac{i\Pi_{\gamma}(\lambda)\ol{\Pi_{\gamma}(z_0)}}{\lambda - \ol{z_0}} = \frac{i\Pi(\lambda)\ol{\Pi(z_0)}}{\lambda - \ol{z_0}}, \quad \lambda \in \Cm_+.
		\end{gather*}
		Therefore $\Pi(\lambda) = \frac{(\lambda - \ol{z_0})F(\lambda)}{i \ol{\Pi(z_0)}}$ is entire with $\Pi(\ol{z_0}) = 0$ as claimed. Furthermore,  for every $\lambda\in \Cm$ we get the limit relation
		\begin{gather}
			\label{eq: limit relation decaying solution}
			\lim_{r\to\infty}P_*(r, \lambda) = \Pi_{\gamma}(\lambda).
		\end{gather}
		Lemma \ref{Lem: convergence implies a.c.} then implies that $\sigma$ is absolutely continuous. 
		The estimate we used to establish the convergence of the integral in \eqref{eq: definition of F} is uniform in $\{\Im \lambda > -\delta\}$ for every $\delta \ge 0$. Therefore $F$ is bounded in every upper horizontal half-plane. In particular, $F$ is bounded on $\R$. Consider the set
		\begin{gather}\label{eq: M set}
			M = \{x\in \R\colon |\Pi(x)| \le 2\}.
		\end{gather}
		We have $|F(z)|\ls |(z - \ol{z_0})^{-1}|$ on $M$ therefore $\|F\|_{L^2(M)} < \infty$.  By Theorem \ref{thm: szego to hardy} there exists $h\in H^2(\Cm_+)$ such that $\Pi_{\gamma}^{-1} = 1 + h$.  If $x\notin M$ then $|\Pi(x)| > 2$ and $|h(x)|> 1/2$. Consequently, the Lebesgue measure of the set $\R\setminus M$ is bounded by $\|h\|_{L^2(\R)}$ by the Chebyshev inequality. Therefore $\|F\|_{L^2(\R\setminus M)}\ls \|F\|_{L^{\infty}(\R)} < \infty$.
		Hence $F\in L^2(\R)$ and 
		\begin{gather*}
			\frac{\Pi(z)}{z - \ol{z_0}} = \frac{1}{i\Pi(z_0)}F(z)\in L^2(\R).
		\end{gather*}
		To prove $\frac{\Pi(z)}{z - \ol{z_0}} \in \dfa$ we need to show that $\F F$ decays very rapidly. 
		We have 
		\begin{gather*}
			F(z) =  \int_0^{r} P(s,z)\ol{P(s,z_0)}\,ds + \int_r^{\infty} P(s,z)\ol{P(s,z_0)}\,ds.
		\end{gather*}
		The first term is the function $k_{r, z_0}(z)\in PW_{[0,r]}$, recall \eqref{eq: reproducing kernel in PW}. Let %$\phi_r(x) = \F\left(\int_r^{\infty} P(s,x)\ol{P(s,z_0)}\,ds\right)$.
		$f_r$ be the inverse Fourier transform of the second term. We know that $f_r$ and $\F F$ coincide on $[r,\infty)$ hence
		\begin{gather*}
			\|\F^{-1} F\|_{L^2[r, +\infty)} = \|f_r\|_{L^2[r, +\infty)} \le \|f_r\|_{L^2(\R)}= \left\|\int_{r}^{\infty} P(x, z)\ol{P(x, z_0)}\, dx\right\|_{L^2(\R)}.
		\end{gather*}
		By the argument similar to the one we used to estimate $\|F\|_{L^{\infty}(\R)}$ we get 
		\begin{gather*}
			\left\|\int_{r}^{\infty} P(x, z)\ol{P(x, z_0)}\, dx\right\|_{L^{\infty}(\R)} = \eps_{\alpha}(r), \quad r\to\infty.
		\end{gather*}
		Let $M$ be as in \eqref{eq: M set}. The Lebesgue measure of $\R\setminus M$ is finite hence 
		\begin{gather*}
			\left\|\int_{r}^{\infty} P(x, z)\ol{P(x, z_0)}\, dx\right\|_{L^{2}(\R\setminus M)} \ls\left\|\int_{r}^{\infty} P(x, z)\ol{P(x, z_0)}\, dx\right\|_{L^{\infty}(\R)} = \eps_{\alpha}(r).
		\end{gather*}
		On the other hand, on $M$ we have $dz\ls |\Pi|^{-2}dz=d\sigma(z)$ hence
		\begin{gather*}
			\left\|\int_{r}^{\infty} P(x, z)\ol{P(x, z_0)}\, dx\right\|_{L^{2}( M)}\ls \left\|\int_{r}^{\infty} P(x, z)\ol{P(x, z_0)}\, dx\right\|_{L^{2}(\R, \sigma)}
			\\
			= \sqrt{2\pi}\left\|\mathcal{O}\left(\mathbf{1}_{[r, \infty)
			}\ol{P(x, z_0)}\right)\right\|_{L^{2}(\R, \sigma)} = \sqrt{2\pi}\left\|\mathbf{1}_{[r, \infty)
			}\ol{P(x, z_0)}\right\|_{L^{2}(\R)} = \eps_{\alpha}(r),
		\end{gather*}
		by the isometry property \eqref{isometry property of spectral measure} of the spectral measure applied for $f(x) = \mathbf{1}_{[r,\infty)}\ol{P(x, z_0)}$. This finishes the first part of the proof of the proposition.
		\medskip
		
		Now we focus on the rate of convergence of $\int_r^{\infty}a(t)\,dt$. Differential equation in the Krein system \eqref{Krein system} for $P_*$ and \eqref{eq: limit relation decaying solution} give
		\begin{gather*}
			\Pi_{\gamma}(\lambda) - P_*(r, \lambda) = -\int_r^{\infty} a(x) P(x,\lambda)\, dx, \quad \lambda\in \Cm.
		\end{gather*}
		For $\lambda = z_0$ this becomes
		\begin{gather}
			\label{eq: P_* - Pi cauchy-Schwarz}
			|P_*(r,z_0) - \Pi_{\gamma}(z_0)| = \left| \int_{r}^{\infty} a(x) P(x,z_0) \, dx \right| \le\norm[a]_{L_2(\R_+)}\cdot\sqrt{\int_{r}^{\infty}|P(x, z_0)|^2\,dx}  = \eps_{\alpha}(r).
		\end{gather}
		Previously we have proved $\Pi_\gamma(\ol{z_0}) =\Pi(\ol{z_0}) = 0$ hence
		\begin{gather*}
			P_*(r,\ol{z_0}) = P_*(r,\ol{z_0}) - \Pi_{\gamma}(\ol{z_0}) = \int_r^{\infty} a(x) P(x,\ol{z_0})\, dx.
		\end{gather*}
		Applying the reflection formula \eqref{Krein system reflection formula} we get 
		\begin{align*}
			P_*(r,\ol{z_0}) &= \int_r^{\infty} a(x) e^{ix\ol{z_0}}\ol{P_*(x,z_0)}\, dx
			\\
			&= \ol{\Pi_{\gamma}(z_0)}\int_r^{\infty} a(x)e^{ix\ol{z_0}} \, dx + \int_r^{\infty} a(x) e^{ix\ol{z_0}}\left[\ol{P_*(x,z_0)} - \ol{\Pi_{\gamma}(z_0)}\right]\, dx.
		\end{align*}
		The second integral is absolutely convergent and is $\eps_{\alpha}(r)$ by \eqref{eq: P_* - Pi cauchy-Schwarz} and the Cauchy-Schwarz inequality. The reflection formula \eqref{Krein system reflection formula} implies $|P_*(r,\ol{z_0})| = \left|e^{i \ol{z_0} r}\ol{P(r,z_0)}\right| = \eps_{\alpha}(r)$ therefore the improper integral  $\int_r^{\infty} e^{ix\ol{z_0}}a(x) \, dx$ converges and 
		\begin{align*}
			\left|\int_r^{\infty} e^{ix\ol{z_0}}a(x)\, dx\right|&\le \frac{|P_*(r,\ol{z_0})|}{|\Pi_{\gamma}(z_0)|} 
			+ \frac{1}{|\Pi_{\gamma}(z_0)|}\left|\int_r^{\infty} a(x) e^{ix\ol{z_0}}\left[\ol{P_*(r,z_0)} - \ol{\Pi_{\gamma}(z_0)}\right] dx\right| = \eps_{\alpha}(r).
		\end{align*}
		Let $A(r) = \int_r^{\infty} e^{ix\ol{z_0}}a(x)\, dx =\eps_\alpha(r)$. We have 
		\begin{gather*}
			\int_r^{\infty}a(x)\,dx = \int_r^{\infty} e^{ix\ol{z_0}}a(x)\cdot e^{-ix\ol{z_0}}\,dx = -A(x)e^{-ix\ol{z_0}}\Big|_r^{\infty} - i\ol{z_0}\int_r^{\infty} A(x)\cdot e^{-ix\ol{z_0}}\,dx.
		\end{gather*}
		Both terms in the right-hand side of the equality are $\eps_{\alpha}(r)$. Therefore $a$ is rapidly oscillating and $a\in \osc_{\alpha}$.
	\end{proof}
	
	\subsection{Assertion \ref{cond: main thm 1}. Krein system with oscillating potential}
	\begin{Prop}\label{prop: conv of p star and p}
		If $a\in \osc_{\alpha}$ for some $\alpha\ge 1$ then $\Pi$ extends analytically into the whole complex plane $\Cm$ and for every $z\in \Cm$ we have
		\begin{gather*}
			|P_*(r, z) - \Pi(z)| = (1 + |z|)\eps_{\alpha}(r)
		\end{gather*}
		uniformly in the strip $\mathcal{U}_{\delta} =\{z\colon \delta > \Im z > -\delta\}$ for every $\delta > 0$. Moreover, if $\alpha > 1$ then the order of $\Pi$ is not greater than $\alpha^* = \frac{\alpha}{\alpha - 1}$.
	\end{Prop}
	\begin{proof}
		Fix some $\delta > 0$. Take a point $z\in\mathcal{U}_{\delta}$ and two positive numbers $r_1 > r$. Using differential equation from Krein system \eqref{Krein system} for $P_*(r, z)$, the reflection formula \eqref{Krein system reflection formula} and  differential equation for $P_*(r, \ol{z})$ one more time we get
		\begin{align*}
			P_*(r_1, z) - P_*(r, z) &= -\int_{r}^{r_1} a(t) P(t,z)\,dt= -\int_{r}^{r_1} a(t) e^{itz}\ol{P_*(t,\ol{z})}\,dt
			\\
			&=- \int_{r}^{r_1}a(t)e^{itz} \ol{\left[1 - \int_0^ta(s)P(s,\ol{z})\,ds \right]}dt
			\\
			&= - \int_{r}^{r_1}a(t) e^{itz}dt + \int_0^{r_1}\ol{a(s)P(s,\ol{z})}\left[\int_{\max(r,s)}^{r_1}a(t)e^{itz}dt\right]ds.
		\end{align*}
		Therefore we have 
		\begin{gather}
			\label{eq: P_*(r_2, z) - P_*(r_1, z) bound}
			|P_*(r_1, z) - P_*(r, z)|\le \sup_{s\in[r, r_1]}\left|\int_s^{r_1}a(t)e^{itz}dt\right|\cdot\left(1 + \int_0^{r_1}|a(s)P(s,\ol{z})|\, ds\right).
		\end{gather}
		Let $A(r) = -\int_r^{\infty} a(t)\, dt =\eps_{\alpha}(r)$. We have
		\begin{gather*}
			\int_s^{r_1} a(t) e^{itz}dt = A(t)e^{itz}\Big|_s^{r_1} - iz \int_s^{r_1} A(t) e^{itz}\,dt,
			\\
			\sup_{s\in [r, r_1]}\left|\int_s^{r_1} a(t) e^{itz}dt\right|\le (2 + |z|(r_1 - r))e^{r_1\delta}\sup_{s\ge r}|A(s)|.
		\end{gather*}
		To estimate the second integral in \eqref{eq: P_*(r_2, z) - P_*(r_1, z) bound} we use the Cauchy-Schwarz inequality. It gives
		\begin{gather*}
			\int_0^{r_1}|a(s)P(s,\ol{z})|\, ds \le \|P(s, \ol{z})\|_{L^2([0,r_1])}\|a\|_{L^2(\R_+)}.
		\end{gather*}
		Next, we use formula \eqref{Krein system reflection formula} and and Proposition \ref{Prop: Krein system estimates} to write
		\begin{gather*}
			|P(s,\ol{z})|= |e^{is\ol{z}}P_*(s, z)|\le e^{s(\delta + \|a\|_{L^2(\R_+)})}.
		\end{gather*}
		Therefore $\int_0^{r_1}|a(s)P(s,\ol{z})|\, ds\ls e^{r_1(\delta + \|a\|_{L^2(\R)})}$. If we substitute the obtained bounds into \eqref{eq: P_*(r_2, z) - P_*(r_1, z) bound} and additionally assume $r_1 - r\le 1$ then it will become
		\begin{gather}
			\label{eq: sup A(s)}
			|P_*(r_1, z) - P_*(r, z)|\ls (1 + |z|)\exp\left(2r\delta + r\|a\|_{L^2(\R_+)}\right)\sup_{s\ge r}|A(s)|.
		\end{gather}
		Uniformly in $\mathcal{U}_{\delta}$ for $r_1 \le r_2\le r_1 + 1$ we have $|P_*(r_1, z) - P_*(r, z)|\le (1 + |z|)\ea(r)$ hence $P_*(r,z)$ converges as $r\to\infty$ very rapidly on compact subsets of $\Cm$. This limit coincides with $\Pi_{\gamma}$ in $\Cm_+$ hence $\Pi$ has an entire continuation into the whole complex plane $\Cm$. Now we have to bound the order of $\Pi$ when $\alpha > 1$.  
		Recall \eqref{eq: sup A(s)}. For $z\in \mathcal{U}_\delta$ we have the uniform bound
		\begin{gather*}
			|\Pi(z) - 1|\le \sum_{n\ge 0}|P_*(n + 1, z) - P_*(n,z)|\ls (1 + |z|)\sum_{n\ge 0}e^{n\left(2\delta + \|a\|_{L^2(\R_+)}\right)}B(n),
		\end{gather*}
		where $B(r) = \sup_{s\ge r}|A(s)| = \ea(r)$ and the constant in $\ls$ depends only on $\|a\|_{L^2(\R_+)}$. Inequality $|\Pi(z)|\ls \exp(c\delta^{\alpha^*})\le\exp(c|z|^{\alpha^*})$ in $\mathcal{U}_{\delta}$ with some constant $c$ then follows from Lemma \ref{Lem: order of a sum bound}. To conclude the proof notice that from Proposition \ref{Prop: Krein system estimates} and Lemma \ref{Lem: convergence in Krein theorem} we know that $\Pi$ is bounded in the half-plane $\{\Im z \ge \delta\}$.
	\end{proof}
	The estimate in the previous proposition implies $|\Pi(z)|\ls 1 + |z|$ uniformly in $\mathcal{U}_{\delta}$ for every $\delta > 0$. This inequality can be strengthened in the following way.
	
	\begin{Cor}
		Assume that $a\in \osc_{\alpha}$ for some $\alpha \ge 1$ and let $\delta,\beta> 0$ be positive numbers. Then we have $|\Pi(z)|\ls 1 + |z|^{\beta}$ uniformly in $\Omega_{\delta} = \{z\colon \Im z > -\delta\}$.  
	\end{Cor}
	\begin{proof}
		From Proposition \ref{Prop: Krein system estimates} and Lemma \ref{Lem: convergence in Krein theorem} we know that $\Pi$ is bounded in $\{\Im z \ge 1\}$ hence we need to show $|\Pi(z)|\ls 1 + |z|^{\beta} $ only for the strip $S_{\delta} = \{z\colon -\delta \le \Im z \le 1\}$. 
		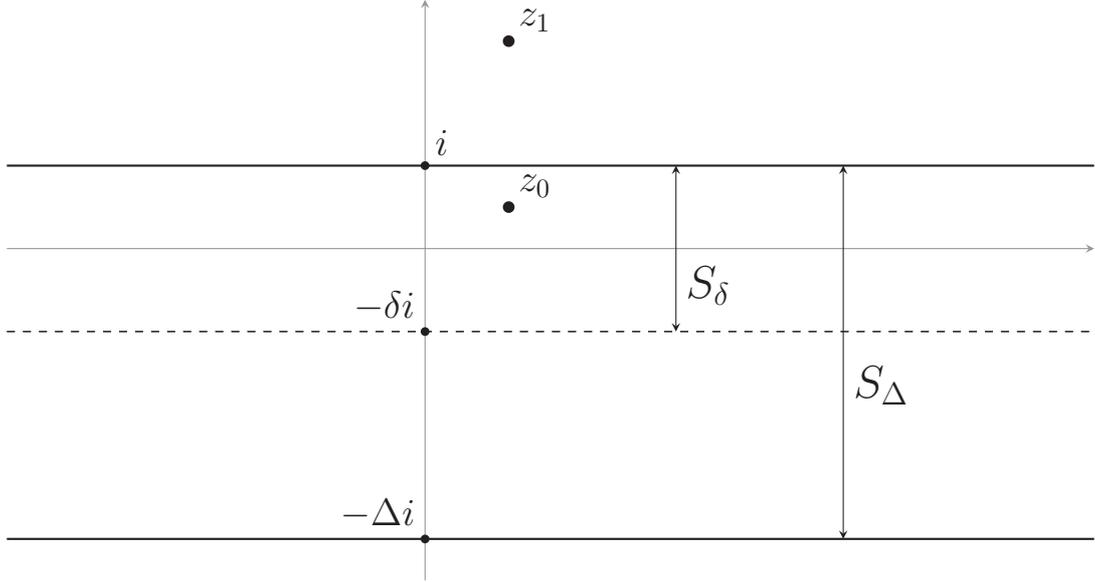
\begin{figure}[ht]
			\centering
			\tikzset{every node/.style={font=\scriptsize}}
			\begin{tikzpicture}[scale=1.1][every node/.append style={midway}]
				
				\coordinate (E) at (2,0.5);
				% \draw[rotate around={-60:(E)}, pattern=horizontal lines light gray,label=above right:$K$] (E) ellipse (40pt and 80pt);
				% \coordinate[label=above right:\Large $K$] (E) at (E);
				
				\coordinate (O) at (0,0);
				\coordinate (A) at (-3,-2);
				\coordinate (B) at (5,-2);
				\coordinate (C) at (5,4);
				\coordinate (D) at (-3,4);
				\coordinate (X) at (10,0);
				
				\draw [-stealth, gray](-5,0) -- (8,0);
				\draw [-stealth, gray](0,-4) -- (0,3);
				
				\coordinate (y1) at (0,1);
				\coordinate[label=above right:\Large$i$] (y1) at (y1);
				\fill[black] (y1) circle (1.5pt);
				
				\coordinate (z_0) at (1,0.5);
				\coordinate[label=above right:\Large$z_0$] (z_0) at (z_0);
				\fill[black] (z_0) circle (2pt);
				\draw[black, dashed] (-5, -1) -- (8, -1);
				
				\draw[black, dashed] (-5, -1) -- (8, -1);

				\coordinate (z_1) at (1,2.5);
				\coordinate[label=above right:\Large$z_1$] (z_1) at (z_1);
				\fill[black] (z_1) circle (2pt);

				\draw[black, thick] (-5, -3.5) -- (8, -3.5);
				\draw[black, thick] (-5, 1) -- (8, 1);

				\coordinate (Deltah) at (0, -3.5);
				\coordinate[label=above left:\Large$-\Delta i$] (Deltah) at (Deltah);
				\fill[black] (Deltah) circle (1.5pt);
				
				\coordinate (deltah) at (0, -1);
				\coordinate[label=above left:\Large$-\delta i$] (deltah) at (deltah);
				\fill[black] (deltah) circle (1.5pt);
				
				\coordinate[label=above right:\LARGE $S_\delta$] (3,-0.8) at (3,-0.8);
				
				\coordinate[label=above right:\LARGE $S_\Delta$] (5,-2) at (5,-2);                
				\draw [stealth-stealth](3,1) -- node[right] {} (3,-1);
				\draw [stealth-stealth](5,1) -- node[right] {}  (5,-3.5);
				
			\end{tikzpicture}
			\caption{Strip for the Hadamard three lines theorem.}
			\label{fig:three lines}
		\end{figure}
		
		Take large $\Delta > 0$ such that $\frac{1 + \delta}{1 + \Delta} \le \beta$ and let $S_{\Delta} = \{z\colon -\Delta \le \Im z \le 1\}$ be the strip similar to $S_{\delta}$. we have $\partial S_{\Delta} = L_1 \cup L_2$, where $L_1 = \{\Im z = 1\}$ and $L_2 = \{\Im z = - \Delta\}$.  We want to apply the Hadamard three lines theorem, see page 33 in \cite{Reed1975MethodsOM}: we know that $\Pi$ is bounded on $L_1$ and $|\Pi(z)|\ls 1 + |z|$ uniformly in $S_{\Delta}$. 
		Take $z_0\in S_{\Delta}$ and let $z_1 = \ol{z_0} + 3i$, $ F(z) = \frac{\Pi(z)}{z - z_1}$, see Figure \ref{fig:three lines}.
		We have $|z - z_1|\ge 1$ for $z\in L_1$ hence
		\begin{gather*}
			\sup_{z\in L_1}|F(z)| = \sup_{z\in L_1}\frac{|\Pi(z)|}{|z - z_1|} \le \sup_{z\in L_1}|\Pi(z)|\ls 1.
		\end{gather*}
		Further, if $z\in S_{\Delta}$ then we can write $\frac{|z| + 1}{|z - z_1|}\le 1 + \frac{|z_1| + 1}{|z - z_1|}\ls 1 + |z_1|\ls 1 + |z_0|$ and
		\begin{gather*}
			\sup_{z\in S_{\Delta}}|F(z)| = \sup_{z\in S_{\Delta}}\frac{|\Pi(z)|}{|z - z_1|}\le (1 + |z_0|)\sup_{z\in S_{\Delta}}\frac{|\Pi(z)|}{|z| + 1} \ls 1 + |z_0|
		\end{gather*}
		uniformly for $z_0\in S_{\delta}$.
		Now the Hadamard three lines theorem implies 
		\begin{gather*}
			|F(z_0)| \le \left(\sup_{z\in L_1}|F(z)|\right)^{1 - h}\cdot \left(\sup_{z\in L_2}|F(z)|\right)^{h} \ls 1 + |z_0|^{h},
		\end{gather*}
		where $h = \frac{1 - \Im z_0}{\Delta + 1} \le \frac{1 + \delta}{1 + \Delta} \le \beta$ due to the choice of $\Delta$. This gives
		\begin{gather*}
			|\Pi(z_0)| = (2 + 2|\Im z_0|)|F(z_0)|\ls 1 + |z_0|^{h}\ls 1 + |z_0|^{\beta}
		\end{gather*}
		uniformly for $z_0\in S_{\delta}$. The proof is concluded.
		
	\end{proof}
	
	\subsection{Assertion \ref{cond: main thm 2}. Krein system with entire inverse \Szego function}
	
	\begin{Lem}\label{Prop Szego function has a zero}
		Assume that $a\in L^2(\R_+)$, $\sigma$ is absolutely continuous and $\Pi$ is entire of finite order. Then either $\Pi$ has at least one zero in $\Cm$ or $a\equiv 0$ in $L^2(\R_+)$.
	\end{Lem}
	\begin{proof}
		Assume that $\Pi$ does not have any zeroes in $\Cm$. Then $\Pi(z) = e^{g(z)}$, where $g$ is a polynomial. Let $\gamma$ be as in \eqref{eq: tmp convergence in Krein theorem}. From Theorem \ref{thm: szego to hardy} we know that $e^{g(z) + i\gamma} = \Pi_\gamma(z)\to 1$ as $\Im z\to\infty$. It is possible only when $g(z) = -i\gamma$  and $\Pi_\gamma = 1$ are constants in $\Cm_+$.  In this case $\sigma$ coincides with the Lebesgue measure and therefore $a\equiv 0$ in $L^2(\R_+)$.
		% Assume the contrary then $\Pi(z) = e^{g(z)}$ for some polynomial $g$. We know that $\Pi$ is bounded in $\{\Im z > 1\}$ hence $\Re g$ is bounded there. It follows that $g$ is linear, $g=az + b$ with imaginary $a$. Also we know that $\Pi_{\gamma}(iy)$ converges to $1$ as $y\to+\infty$ (this follows e.g. from Theorem \ref{thm: szego to hardy} or the explicit representation of $\Pi$) hence $a = 0$ and $b = 1$, $\Pi$ is constant and $a = 0$, which is a contradiction.
	\end{proof}
	
	The idea of the proof of the following proposition is similar to the idea used in Theorem \ref{thm: MNT with order}, it was previously implemented in Lemma 4.2 from \cite{Gub2024} in a slightly different situation with more technical details.
	\begin{Thm}\label{thm: from szego function}
		Assume that $a\in L^2(\R_+)$, $\sigma$ is a.\,c., $\Pi$ is entire with $\Pi(\ol{z_0}) = 0$ for some $z_0\in \Cm_+$ and $\frac{\Pi}{z - \ol{z_0}}\in \dfa$. Then we have $P(r, z_0) = \eps_{\alpha}(r)$.
	\end{Thm}
	\begin{proof}
		Let $\phi = \F\left(\frac{\Pi}{z - \ol{z_0}}\right)$ and  $G, G_r$ be defined as
		\begin{gather*}
			G(z) =\frac{\Pi(z)}{z-\ol{z_0}} = \frac{1}{\sqrt{2\pi}}\int_0^{\infty}\phi(t)e^{itz }\,dt ,\qquad G_{r}(z) = \frac{1}{\sqrt{2\pi}}\int_0^{r}\phi(t)e^{itz}\,dt,\quad z\in\Cm.
		\end{gather*}
		Recall the definition \eqref{definition of the minimization function for Krein system} of $\bm_r$. We have $G_r\in PW_{[0,r]}$ and $d\sigma(t) = |\Pi(t)|^{-2}dt$ hence
		\begin{gather}
			\label{bound for m_r in terms of G}
			\bm_r(\sigma,z_0)\le \frac{1}{2\pi}\norm[G_r/G_r(z_0)]_{L^2(\R,\sigma)}= \frac{1}{2\pi|G_{r}(z_0)|^2}\int_{-\infty}^{\infty}\frac{|G_{r}(t)|^2}{|\Pi(t)|^2}\, dt.
		\end{gather}
		Let us examine the right-hand side of the last inequality. For $z\in\Cm$, we have
		\begin{gather*}
			G(z) - G_r(z) = \frac{1}{\sqrt{2\pi}}\int_{r}^{\infty}\phi(t)e^{itz }\,dt.
		\end{gather*}
		Consequently $\|G - G_r\|_{L^2(\R)} =  \|\phi\|_{L^2[r, +\infty)} = \eps_{\alpha}(r)$ and for $z\in \R$ we can write
		\begin{align*}
			%\label{eq: difference of G  in the point}
			\left|G(z) - G_r(z)\right| &\ls \int_{r}^{\infty}\left|\phi(t)e^{itz }\right|\,dt = \|\phi\|_{L^1[r, +\infty)} = \eps_{\alpha}(r).
		\end{align*}
		Therefore 
		\begin{gather}
			\label{pointwise bound on difference G}
			\frac{1}{|G_{r}(z_0)|^2} - \frac{4(\Im z_0)^2}{|\Pi(z_0)|^2} = \frac{1}{|G_{r}(z_0)|^2}- \frac{1}{|G(z_0)|^2} = \eps_{\alpha}(r).
		\end{gather}
		Hence the first multiplier in \eqref{bound for m_r in terms of G} converges very rapidly, 
		for the integral in \eqref{bound for m_r in terms of G} we have
		\begin{align*}
			\int_{-\infty}^{\infty}\frac{|G_{r}(t)|^2}{|\Pi(t)|^2}\, dt &= \int_{-\infty}^{\infty}\left|\frac{G(t)}{\Pi(t)} + \frac{G_{r}(t) - G(t)}{\Pi(t)}\right|^2\, dt
			= \int_{-\infty}^{\infty}\left|\frac{1}{t - \ol{z_0}} + \frac{G_{r}(t) - G(t)}{\Pi(t)}\right|^2\, dt 
			\\
			&= \int_{-\infty}^{\infty}\left(\frac{1}{|t - \ol{z_0}|^2} + 2\Re\left(\frac{1}{t - z_0}\cdot\frac{G_{r}(t) - G(t)}{\Pi(t)} \right) + \left|\frac{G_{r}(t) - G(t)}{\Pi(t)}\right|^2\right)\, dt. 
		\end{align*}
		We have $\|G - G_r\|_{L^2(\R)}  = \eps_{\alpha}(r)$ and
		\begin{gather*}
			\norm[\frac{1}{(t - z_0)\Pi(t)}]_{L^2(\R)} \ls \sqrt{\int_{\R}\frac{d\sigma(t)}{1 + t^2}}< \infty
		\end{gather*}
		therefore 
		\begin{gather*}
			\left|\int_{-\infty}^{\infty}\frac{1}{t - z_0}\cdot\frac{G_{r}(t) - G(t)}{\Pi(t)} \, dt\right|\ls \norm[\frac{1}{(t - z_0)\Pi(t)}]_{L^2(\R)}\cdot\norm[G - G_r]_{L^2(\R)} = \eps_{\alpha}(r).
		\end{gather*}
		Furthermore, Theorem \ref{thm: szego to hardy} states that $\Pi^{-1}\gamma^{-1} = \Pi_{\gamma}^{-1} =1 + h$ with $h\in H^2(\Cm_+)$, hence
		\begin{gather*}
			\norm[\frac{G_{r}(t) - G(t)}{\Pi}]^2_{L^2(\R)}\ls \norm[G - G_r]^2_{L^2(\R)} + \norm[G - G_r]^2_{L^\infty(\R)} \cdot \|h\|^2_{H^2(\Cm_+)} = \eps_{\alpha}(r).
		\end{gather*}
		It follows that
		\begin{gather}
			\label{eq: entire function integral bound}
			\int_{-\infty}^{\infty}\frac{|G_{r}(t)|^2}{|\Pi(t)|^2}\, dt = \int_{-\infty}^{\infty}\frac{dt}{|t - \ol{z_0}|^2} + \eps_{\alpha}(r) = \frac{\pi}{\Im z_0} + \eps_{\alpha}(r).
		\end{gather}
		Substituting \eqref{pointwise bound on difference G} and \eqref{eq: entire function integral bound} into \eqref{bound for m_r in terms of G} we get
		\begin{align*}
			\bm_r(z_0) 
			=  \left(\frac{4(\Im z_0)^2}{|\Pi(z_0)|^2} + \eps_{\alpha}(r) \right)\left(\frac{1}{2\Im z_0} +\eps_{\alpha}(r) \right) = \frac{2\Im z_0}{|\Pi(z_0)|^2} + \eps_{\alpha}(r).
			%(r)\stackrel{\eqref{m infinity formula} }{=}\bm_{\infty}(z_0) + \eps_{\alpha}(r).
		\end{align*}
		Now \eqref{expression for the szego function} and \eqref{m_r function formula} imply
		\begin{align*}
			\bm_r(z_0) - \frac{2\Im z_0}{|\Pi(z_0)|^2} = \left(\int_0^{r}|P(t, z_0)|^2\,dt\right)^{-1} - \left(\int_0^{\infty}|P(t, z_0)|^2\,dt\right)^{-1}
			\gs\int_r^{\infty}|P(t, z_0)|^2\,dt.
		\end{align*}
		Thus $\|P(t, z_0)\|_{L^2[r, \infty)} = \eps_{\alpha}(r)$. Recall the differential equation for $P(r, z_0)$ from Krein system \eqref{Krein system}: $P'(r,z_0) = iz_0 P(r,z_0) - \ol{a(r)}P_*(r,z_0)$.
		% \begin{gather*}
			%     %\frac{\partial}{\partial r}
			%     P'(r,z_0) = iz_0 P(r,z_0) - \ol{a(r)}P_*(r,z_0).
			% \end{gather*}
		From Proposition \ref{Prop: Krein system estimates} we know that $P_*(r, z_0)$ is bounded in $r$ hence $P'(r,z_0)\in L^2(\R_+)$ and therefore
		\begin{align*}
			|P(r, z_0)|^2 &= \left|2\int_r^{\infty}\Re\left(P(t, z_0)P'(t,z_0)\right)\, dt\right| 
			\\
			&\le 2\|P(t, z_0)\|_{L^2[r, +\infty)}\left\|P'(t,z_0)\right\|_{L^2[r, +\infty)} = \eps_{\alpha}(r).
		\end{align*}
		This concludes the proof.
	\end{proof}
	\subsection{Proof of Theorem \ref{thm: main equivalences of convergences}}
	
	\begin{proof}[Proof of Theorem \ref{thm: main equivalences of convergences}]
		Implications $\ref{cond: main thm 3}\Longrightarrow \ref{cond: main thm 1}$ and $\ref{cond: main thm 3}\Longrightarrow \ref{cond: main thm 2}$ are proved in Proposition \ref{prop: from small p} for $\alpha \ge 1$. From Theorems \ref{thm: entropy and variation} and \ref{thm: osc var} we know that \ref{cond: main thm 1} is equivalent to $E_a(r) = \ea(r)$. Hence the equivalence \ref{cond: main thm 1}$\Longleftrightarrow$\ref{cond: main thm C} follows from Theorem \ref{thm: bessonov, denisov, entropy}. Also, when $\alpha = 1$, \ref{cond: main thm 1}$\Longleftrightarrow$\ref{cond: main thm 2} immediately follows from Theorem \ref{thm: reformulation master}. Thus, the theorem is proved for $\alpha = 1$ and for $\alpha > 1$ we need to show $\ref{cond: main thm 1}\Longrightarrow \ref{cond: main thm 3}$ and $\ref{cond: main thm 2}\Longrightarrow \ref{cond: main thm 3}$.

		If \ref{cond: main thm 1} holds then Proposition \ref{prop: conv of p star and p} applies and $\Pi$ is entire of finite order. Next, by Lemma \ref{Prop Szego function has a zero}, $\Pi(\ol{z_0}) = 0$ for some $z_0\in \Cm_+$ and again by Proposition \ref{prop: conv of p star and p}
		\begin{gather*}
			|P_*(r, \ol{z_0})| = |P_*(r, \ol{z_0}) - \Pi(\ol{z_0})| = \ea(r).
		\end{gather*}
		Hence \eqref{Krein system reflection formula} gives $P(r, z_0) = e^{i z_0 r}\ol{P_*(r,\ol{ z_0})} = \eps_{\alpha}(r)$, which is exactly \ref{cond: main thm 3}.
		
		Assume that $\Pi$ satisfies assertion \ref{cond: main thm 2}. Then, by Lemma \ref{Lem: salpha description}, $\Pi$ is entire of finite order and, by Lemma \ref{Prop Szego function has a zero}, it has some zero $\ol{z_0}$. Proposition \ref{prop: salpha does not depend on z_0} gives $\frac{\Pi(z)}{z - \ol{z_0}}\in\dfa$ and from Theorem \ref{thm: from szego function} we get $P(r, z_0) = \eps_{\alpha}(r)$. This finishes the implication  $\ref{cond: main thm 2}\Longrightarrow\ref{cond: main thm 3}$ and the proof of the whole theorem for $\alpha > 1$.\medskip
		
		% When $\alpha  = 1$ the previous proof does not work because Proposition \ref{Prop Szego function has a zero} fails, i.e. it is possible that the inverse \Szego function of infinite order does not have any zeroes. However, Theorem \ref{thm: main equivalences of convergences} for $\alpha = 1$ immediately follows from Theorems \ref{thm: entropy and variation}, \ref{thm: osc var} and \ref{thm: reformulation master}:
		% \begin{gather*}
			% 	\frac{\Pi - \Pi(z_0)}{z - z_0}\in \df_1
			% 	\stackrel{\text{Th } \ref{thm: reformulation master}}{\Longleftrightarrow} 
			% 	E_a(r) = \eps_1(r) 
			% 	\stackrel{\text{Th } \ref{thm: entropy and variation}}{\Longleftrightarrow} 
			% 	D_a(r) = \eps_1(r)
			% 	\stackrel{\text{Th } \ref{thm: osc var}}{\Longleftrightarrow} 
			% 	a\in \osc_1.
			% \end{gather*}
		
	\end{proof}
	
	\section{Entropy estimations. Proofs of Theorems \ref{thm: entropy and variation} and \ref{thm: osc var}}\label{section: second theorem}
	\subsection{Oscillation and variation. Proof of Theorem \ref{thm: osc var}}
	For a function $F\in L^2([0,1])$ we let $C_F = \int_0^1 F(s)\, ds$ and
	\begin{gather*}
		\var(F) = \int_0^1 F(s)^2\,ds - C_F^2 = \int_0^1(F(s) - C_F)^2\,ds
	\end{gather*}
	be the mean value and the variation of $F$. Below we will work with the absolutely continuous functions on $[0,1]$ satisfying the assertions
	\begin{gather}\label{eq: epsilon delta assertion on the antiderivative}
		F(0) = 0, \qquad \var(F)\le \eps, \qquad \|F'\|_{L^2([0,1])} \le \delta,
	\end{gather}
	where $\eps$ and $\delta$ are small positive numbers. Let us prove the following technical lemma.
	\begin{Lem}\label{Lemma 1}
		Assume that $F$ satisfies the assertions in \eqref{eq: epsilon delta assertion on the antiderivative} and
		let 
		\begin{gather*}
			\gamma = \gamma(\eps, \delta) = \eps^{1/2} + \eps^{1/4}\delta^{1/2}.
		\end{gather*}
		Then we have the estimates
		\begin{gather*}
			|C_F|\le 2\gamma,
			\qquad
			\sup_{t\in [0,1]}|F(t)|\le 4\gamma,
			\qquad
			\|F^2 - C_F^2\|_{L^2([0,1])} \le 6\eps^{1/2}\gamma.
		\end{gather*}
	\end{Lem}
	\begin{proof}
		The Cauchy-Schwarz inequality gives
		\begin{gather*}%\label{Cauchy trick}
			\int_0^r\left(F(s) - C_F\right)^2ds\int_0^r F'(s)^2\,ds \ge\left(\int_0^r(F(s) - C_F)F'(s)\,ds\right)^2 
			= \left(\frac{F(r)^2}{2} - C_F F(r)\right)^2.
		\end{gather*}
		Hence we have
		\begin{gather}\label{eq: ineq from the C-S}
			|F(r)^2 - 2C_F F(r)| \le 2\sqrt{\eps}\delta\le 2\gamma^2.
		\end{gather}
		Furthermore we can write
		\begin{align*}
			\eps &\ge \var(F) = \int_0^1F(s)^2 \,ds - C_F^2 
			=C_F^2 + \int_0^1[F(s)^2 - 2C_FF(s) ]\,ds.
		\end{align*}
		Rearranging the terms and applying \eqref{eq: ineq from the C-S} we get 
		\begin{gather*}
			C_F^2 \le \eps + \int_0^1|F(s)^2 - 2C_FF(s)| \,ds\le \eps + 2\sqrt{\eps}\delta\le 2\gamma^2.
		\end{gather*}
		The bound $|C_F|\le 2\gamma$ follows. Furthermore, for every $r\in[0,1]$ we have 
		\begin{gather*}
			(F(r) - C_F)^2 = C_F^2 + [F(r)^2 - 2C_F F(r)] \le 4\gamma^2.
		\end{gather*}
		Therefore $|F(r)|\le |C_F| + 2\gamma \le 4\gamma$. 
		Now, the inequality
		\begin{gather*}
			\|F^2 - C_F^2\|_{L^2([0,1])}\le \|F - C_F\|_{L^2([0,1])}(\|F\|_{L^{\infty}([0,1])} + |C_F|)\le \sqrt{\eps}\cdot 6\gamma
		\end{gather*}
		finishes the proof.
	\end{proof}
	\begin{proof}[Proof of Theorem \ref{thm: osc var}]
		Recall \eqref{eq: def variation of a} that we have $g_{a,r}(t) = \int_r^ta(s)\,ds$ and
		\begin{gather*}
			D_a(r) = 2\int_r^{r + 2}|g_{a,r}(t)|^2\,dt - \left|\int_r^{r + 2}g_{a,r}(t)\,dt\right|^2.
		\end{gather*} 
		We need to prove that $a\in \osc_\alpha$ if and only if $D_a(r) = \ea(r)$. If $a\in \osc_\alpha$ then $\sup_{t\ge r}g_{a,r}(t) = \ea(r)$ and consequently $D_a(r) = \ea(r)$, which finishes the ``only if'' part. 
		
		Assume that $D_a(r)=\ea(r)$. For $r\ge 0$ consider the functions $q_r(t) = \Re g_{a,r}(r + 2t)$ and $p_r(t) = \Im g_{a,r}(r + 2t)$ on $[0,1]$. We have $D_a(r) = 4\var(p_r) + 4\var(q_r)$. Hence $\var(p_r) = \ea(r)$ and $\var(q_r) = \ea(r)$. In particular, $\var(p_r), \var(q_r) \to 0$ as $r\to\infty$. 
		Also we have $q_r'(t) = 2\Re a(r + 2t)$ and $p_r'(t) = 2\Im a(r + 2t)$, therefore $\|p_r'\|_{L^2[0,1]}, \|q_r'\|_{L^2[0,1]}\to 0$ as $r\to \infty$. Lemma \ref{Lemma 1} then applies for $p_r$ and $q_r$. It gives 
		\begin{gather*}
			\sup_{s\in [r, r+ 2]}|g_{a,r}(s)| \le \sup_{t\in [0,1]}|p_r(t)| + \sup_{t\in [0,1]}|q_r(t)|\ls \var(p_r)^{1/4} + \var(q_r)^{1/4} = \ea(r).
		\end{gather*}
		The assertion $a\in \osc_\alpha$ follows.
	\end{proof}
	\subsection{Ordered exponential. Reformulation of Theorem \ref{thm: entropy and variation}}
	Recall definition \eqref{eq: entropy def} of the entropy function $E_a$. The matrix $N_a$ is a solution of $N_a'(t) = JQ(t)N_a(t)$ satisfying $N_a(0) = \I$, where $\I$ is the $2\times 2$ identity matrix. Let us study this differential equation in more general form.
	\subsubsection{Ordered exponential}
	Let $A$ be a $2\times 2$ matrix-valued function on $[0,1]$ with entries from $L^1[0,1]$. Define $X_A$ as the solution of
	\begin{gather*}
		X_A'(t) = A(t) X_A(t), \qquad X(0) = \I.
	\end{gather*}
	The matrix $X_A$ is called the ordered exponential of $A$. It admits the following series representation:
	\begin{gather}
		\label{explicit representaion ordered exponent}
		X_A(t) = \I + \sum_{m = 1}^{\infty}\int_0^t A(t_1) \int_0^{t_1} A(t_2)\int_0^{t_2}\ldots \int_0^{t_{m - 1}} A(t_m)\, dt_m\ldots dt_3\,dt_2\, dt_1.
	\end{gather}
	Define the function $F_A$ on $\R$ and its Taylor coefficients $\{a_n\}_{n\ge 0}$ by
	\begin{gather}\label{eq: def F_A}
		F_A(s) = \det\left(\int_0^1X_{sA}(t)X_{sA}^{T}(t)\,dt\right) = \sum_{n\ge 0} a_ns^n.
	\end{gather}
	Assume that $A$ is of the form $A = \left(\begin{smallmatrix}-q & p \\ p & q\end{smallmatrix}\right)$, where $p$ and $q$ are two functions from $L^1([0,1])$. 
	Let $J = \left(\begin{smallmatrix} 0 & 1 \\ -1 & 0\end{smallmatrix}\right)$, we have $J^{-1} = -\left(\begin{smallmatrix} 0 & 1 \\ -1 & 0\end{smallmatrix}\right)$ and $ JAJ^{-1} = -A$.
	Then 
	\begin{gather*}
		(JX_A(t)J^{-1})' = J A(t)X_A(t)J^{-1} =(JA(t)J^{-1})(JX_A(t)J^{-1}) = -A(t)JX_A(t)J^{-1},
	\end{gather*}
	hence $X_{-A}(t) = JX_A(t)J^{-1}$  for every $t$. From formula \eqref{eq: def F_A} we see $F_A(s) = F_{-A}(s)$ for every $s\in \R$. We also have $F_{-A}(s) = F_A(-s)$ hence $F$ is even and 
	$a_{n}= 0$ when $n$ is odd. Recall that for a function $f\in L^2([0,1])$ we use the notation 
	\begin{gather*}
		\var(f) = \int_0^1f(s)^2\,ds - \left(\int_0^1f(s)\,ds\right)^2.
	\end{gather*}
	\begin{Lem}\label{lemma: first coefficient general situation}
		We have $a_2 = 4 \var(g_p) + 4\var(g_q)$, where $g_p(t) = \int_0^tp(x)\,dx$, $g_q(t) = \int_0^t q(x)\,dx$.
	\end{Lem}
	\begin{proof}
		The proof is a calculation. We have
		\begin{align*}
			X_{sA}(t) &= \I + s\int_0^tA(t_1)\,dt_1 + s^2\int_0^t\int_0^{t_1}A(t_1)A(t_2)\,dt_2\,dt_1 + o(s^2),
			\\
			X_{sA}(t)X_{sA}(t)^T &= \I + s\int_0^tA(t_1)\,dt_1 + s\int_0^tA^T(t_1)\,dt_1
			+ s^2 \int_0^tA(t_1)\,dt_1 \int_0^tA^T(t_1)\,dt_1
			\\
			&+ s^2\int_0^t\int_0^{t_1}A(t_1)A(t_2)\,dt_2\,dt_1 + s^2\int_0^t\int_0^{t_1}(A(t_1)A(t_2))^T\,dt_2\,dt_1 + o(s^2).
		\end{align*}
		Since $g_p$ and $g_q$ are antiderivatives of $p$ and $q$ we have 
		\begin{gather*}
			\int_0^tA(t_1)\,dt_1 = \int_0^tA^T(t_1)\,dt_1 = 
			\begin{pmatrix}
				-g_q(t) & g_p(t)
				\\
				g_p(t) & g_q(t)
			\end{pmatrix},
			\\
			\int_0^tA(t_1)\,dt_1\int_0^tA^T(t_1)\,dt_1 = (g_p^2(t) + g_q^2(t))\I.
		\end{gather*}
		Next, we write
		\begin{align*}
			A(t_1)A(t_2) &= 
			\begin{pmatrix}
				-q(t_1) & p(t_1)
				\\
				p(t_1) & q(t_1)
			\end{pmatrix}
			\begin{pmatrix}
				-q(t_2) & p(t_2)
				\\
				p(t_2) & q(t_2)
			\end{pmatrix}
			\\
			&=
			\begin{pmatrix}
				q(t_1)q(t_2) +  p(t_1) p(t_2)  & -q(t_1)p(t_2) + p(t_1)q(t_2)
				\\
				q(t_1)p(t_2) - p(t_1)q(t_2) & q(t_1)q(t_2) +  p(t_1) p(t_2)
			\end{pmatrix}
			.
		\end{align*}
		Therefore $A(t_1)A(t_2) + (A(t_1)A(t_2))^T = 2\big(q(t_1)q(t_2) +  p(t_1) p(t_2)\big) \I$. Also we have 
		\begin{gather*}
			\int_0^t\int_0^{t_1} q(t_1)q(t_2)\,dt_2dt_1 = \frac{g_q(t)^2}{2}, \quad \int_0^t\int_0^{t_1} p(t_1)p(t_2)\,dt_2dt_1 = \frac{g_p(t)^2}{2},
			\\
			\int_0^t\int_0^{t_1}A(t_1)A(t_2) + (A(t_1)A(t_2))^T\,dt_2\,dt_1 
			%\int_0^t\int_0^{t_2}q(t_1)q(t_2) +  p(t_1) p(t_2)\,dt_1\,dt_2
			= (g_q^2(t) + g_p^2(t))\I.
		\end{gather*}
		Hence we have
		\begin{align*}
			X_{sA}(t)X_{sA}(t)^T = \I +
			2s
			\begin{pmatrix}
				-g_q(t) & g_p(t)
				\\
				g_p(t) & g_q(t)
			\end{pmatrix}
			+2s^2 (g_q^2(t) + g_p^2(t)) \I + o(s^2).
		\end{align*}
		Integrating and taking the determinant, we get
		\begin{align*}
			a_2 &= -4\left(\int_0^1 g_p(t)\,dt\right)^2 - 4\left(\int_0^1 g_q(t)\,dt\right)^2 
			\\
			&+ 4\int_0^1g_p(t)^2\,dt + 4\int_0^1g_q(t)^2\,dt = 4\var(g_p) + 4\var(g_q). 
		\end{align*}
	\end{proof}
	\subsubsection{Reformulation of Theorem \ref{thm: entropy and variation}} We see that the function $N_a$ from \eqref{eq: Na def} is an ordered exponential of the matrix function $JQ(t)$. Definitions \eqref{eq: entropy def} and \eqref{eq: def F_A} of $E_a$ and $F_{JQ}$ are similar, the only difference is the length of the integration segment. On $[0,1]$ define $A_r(t) = 2JQ(r + 2t)$. Then we have 
	\begin{gather*}
		(N_a(r + 2t))' = 2N_a'(r + 2t) = 2JQ(r + 2t)N_a(r + 2t) = A_r(t)N_a(r + 2t)
	\end{gather*}
	hence $X_{A_r}(t) = N_a(r + 2t)$ and 
	\begin{gather*}
		F_{A_r}(1) = \det\left(\int_0^1X_{A_r}(t)X_{A_r}^{T}(t)\,dt\right) = \det\left(\frac{1}{2}\int_r^{r + 2}N_a(t)N_a^{T}(t)\,dt\right) = \frac{1}{4}E_a(r) + 1.
	\end{gather*}
	Thus, if we want to estimate $E_a$ we can work with $F_{A_r}(1) - 1$. We have $A_r = \smatrix{-q_r}{p_r}{p_r}{q_r}$ where $p_r = 4\Re a(r + 2t)$ and $q_r(t) = 4\Im a(r + 2t)$, this follows from the definition of $A_r$ and $Q = Q_a$. Hence $D_a(r) = \var(g_p) + \var(g_q)$. 
	Therefore
	Theorem \ref{thm: entropy and variation} follows from  the following result.
	\begin{Thm}\label{thm: ordered exponential}
		Let $A$ be a matrix-valued function on $[0,1]$ of the form $\smatrix{-q}{p}{p}{q}$ with $p,q\in L^2([0,1])$. Let $g_p(t) = \int_0^tp(x)\,dx$ and $g_q(t) = \int_0^t q(x)\,dx$ be the antiderivatives of $p$ and $q$ respectively. Let
		\begin{gather*}
			\var(g_p) + \var(g_q)= \eps, \qquad \|p\|_{L^2([0,1])} + \|q\|_{L^2([0,1])} = \delta.
		\end{gather*}
		Define $F_A$ and $\{a_n\}_{n\ge 0}$ as in \eqref{eq: def F_A} then $\sum_{n\ge 4}|a_n| = o(\eps)$ as $\eps,\delta\to 0$ and consequently $F_A(1) = 1 + a_2 + o(\eps) = 1 + 4\eps + o(\eps)$ as $\eps,\delta\to 0$.
	\end{Thm}
	We will see from the proof of the theorem that the numbers $a_n$ decay very fast and $F_A(s) = 1 + 4\eps s^2 + o(\eps)$ holds for every $s\in \R$.
	
	\subsection{Diagonal case}
	When the matrix $A(t)$ is diagonal, the numbers $a_n$ can be calculated explicitly and Theorem \ref{thm: ordered exponential} can be proved very shortly. 
	\begin{Lem}\label{Lem: coeeficients in the diagonal case}
		If $A(t) = \smatrix{-q}{0}{0}{q}$ is diagonal and $g(t) = \int_0^tq(s)\,ds$ then we have
		\begin{gather*}
			a_n = \frac{2^n}{n!}\int_0^1\int_0^1(g(x) - g(y))^n\,dx\,dy, \qquad n\ge 0.
		\end{gather*}
	\end{Lem}
	\begin{proof}
		For $s\in \R$ we can write
		\begin{gather*}
			X_{sA}(t) = \exp\left(s\int_0^t A(x)dx\right) = 
			\begin{pmatrix}
				e^{-s g(t)} & 0
				\\
				0 & e^{s g(t)}
			\end{pmatrix},
			\\
			F_A(s) = \int_0^1 e^{2sg(t)}\,dt\cdot \int_0^1 e^{-2sg(t)}\,dt.
		\end{gather*}
		Expanding the Taylor series of the exponential, we get
		\begin{align*}
			F_A(s) &=\int_0^1 \sum_{k\ge 0}\frac{(2s)^kg^k(t)}{k!}\,dt \cdot \int_0^1 \sum_{m\ge 0}\frac{(-2s)^mg^m(t)}{m!}\,dt 
			\\
			&= \sum_{k = 0}^{\infty}\sum_{m = 0}^{\infty}\frac{(-1)^m(2s)^{k + m}}{k! m!}\int_0^1g^k(t)\,dt  \int_0^1g^m(t)\,dt.
		\end{align*}
		Changing the order of summation, we get the explicit formula for $a_n$:
		\begin{gather*}
			a_n = 2^n\sum_{l =0}^n\frac{(-1)^l}{l!(n - l)!} \int_0^1g^{n-l}(t)\,dt  \int_0^1g^l(t)\,dt = \frac{2^n}{n!}\int_0^1\int_0^1(g(x) - g(y))^n\,dx\,dy.
		\end{gather*}
		In particular, we see that $a_n = 0$ if $n$ is odd and 
		\begin{gather*}
			a_2 = 2\int_0^1\int_0^1(g(x) - g(y))^2\,dx\,dy = 4\int_0^1g(x)^2\,dx - 4 \left(\int_0^1g(x)\,dx\right)^2 = 4\var(g),
		\end{gather*}
		which is consistent with Lemma \ref{lemma: first coefficient general situation}.
	\end{proof}
	\begin{proof}[Proof of Theorem \ref{thm: ordered exponential} in the diagonal case]
		From the formula established for $a_n$ in Lemma \ref{Lem: coeeficients in the diagonal case} we get
		\begin{gather*}
			|a_n|\le \frac{2^{n-1}|a_2|}{n!}\sup_{x,y\in [0,1]}\left|g(x) - g(y)\right|^{n- 2} 
			\le \frac{2^{n-1}\cdot 4\eps}{n!} \delta^{n - 2},
			%\sup_{x\in [0,1]}|g(x)|^{n- 2}, \quad n\ge 3.
		\end{gather*}
		because $a_2 = 4\eps$ by Lemma \ref{lemma: first coefficient general situation} and $|g(x) - g(y)|\le\|q\|_{L^1([0,1])}\le\|q\|_{L^2([0,1])} =\delta$.
		The estimate $\sum_{n > 2}{a_n} = o(\eps)$ follows.
		%The assertion $\|q\|_{L^2}\le\delta$ implies $|g(x)|\le \delta$ and hence for small $\delta$ we have $\sum_{n > 2}{a_n} = o(\eps)$ as claimed.
	\end{proof}
	\begin{Rema}
		The same argument works  when $A = \smatrix{0}{p}{p}{0}$ or, more generally, when $A(t_1)$ and $A(t_2)$ commute for almost every $t_1, t_2$.
	\end{Rema}
	\subsection{Auxiliary results}
	The proof of Theorem \ref{thm: ordered exponential} in general situation requires technical details. 
	To simplify the exposition we introduce the notation
	\begin{gather}\label{eq: short integral def}
		(f_1\ldots f_n)_t = \int_0^t\ldots\int_0^{t_{n - 1}}f_1(t_1)\ldots f_n(t_n)\,dt_n\ldots\,dt_1,
	\end{gather}
	where we assume  $f_1, \ldots,f_n\in L^1([0,1])$ and $t\in [0,1]$. The following lemma will be used to bound terms with large indexes in \eqref{explicit representaion ordered exponent}.
	
	\begin{Lem}\label{lemma: long product bound}
		Let $f_1,\ldots,f_k\in L^2([0,1])$ be real-valued functions on $[0,1]$ and $F_k(x) = \int_0^x f_k(s)\,ds$ be their antiderivatives for $k = 1,\ldots, n$. Assume that each $F_k$ satisfies assertions from  \eqref{eq: epsilon delta assertion on the antiderivative}. Let $ \gamma$ be as in Lemma \ref{Lemma 1} then we have
		\begin{gather*}
			|(f_1\ldots, f_n)_{t}| \le (8\gamma)^m, %8^m\eps^{m/4}\gamma^m,
			\qquad m = [(n + 1) / 2].
		\end{gather*}
		If $n$ is even and $f_{2i -1} = f_{2i}$ for some $1\le i\le n/2$ then the same inequality holds with $m = n / 2 + 1$. 
	\end{Lem}
	\begin{proof}
		Assume that $n$ is odd, $n = 2m - 1$. We can change the order of integration so that
		\begin{gather*}
			(f_1\ldots f_n)_t = \int_0^{t_0}\,dt_2\int_0^{t_2}\,dt_4\ldots\int_0^{t_{2m - 2}}\,dt_{2m}\left(\prod_{l = 1}^{m-1} f_{2l}\cdot\prod_{k = 1}^{m }\int_{t_{2k}}^{t_{2k - 2}} f_{2k - 1}\,dt_{2k - 1}\right),
		\end{gather*}
		where $t_0 = t$ and $t_{2m} = 0$.
		For every $1\le k \le m$, by Lemma \ref{Lemma 1}, we have 
		\begin{gather}\label{eq: ineq for single inner integral}
			\left|\int_{t_{2k}}^{t_{2k - 2}} f_{2k - 1}(t_{2k - 1})\,dt_{2k - 1}\right| = |F_{2k - 1}(t_{2k - 2}) - F_{2k - 1}(t_{2k})|\le 8\gamma. 
		\end{gather}
		Therefore we can write
		\begin{gather*}
			\left|\prod_{k = 1}^{m }\int_{t_{2k}}^{t_{2k - 2}} f_{2k - 1}(t_{2k - 1})\,dt_{2k - 1}\right|\le  8^m\gamma^m,
			\qquad
			\left|(f_1\ldots f_n)_t\right|\le 8^m\gamma^m\prod_{l = 1}^{m-1} \|f_{2l}\|_{L^1([0,1])}.
		\end{gather*}
		To finish the proof in the case of odd $n$ notice that $\|f_{2l}\|_{L^1([0,1])} = \|F_{2l}'\|_{L^1([0,1])}\le \delta < 1$ by \eqref{eq: epsilon delta assertion on the antiderivative}. When $n = 2m$ is even we proceed similarly: take the outer integrals over $t_2, t_4,\ldots, t_{2m}$ and the inner over $t_1, t_3, \ldots, t_{2m - 1}$.
		
		To obtain sharper inequality for the situation when $n$ is even and $ f_{2i - 1} = f_{2i} = f$ we let the outer integrals be over the variables $t_2,t_4\ldots, t_{2i - 2}$ and $t_{2i + 1},\ldots t_{2m - 1}$.  The product in the inner integral then becomes 
		\begin{align*}
			\prod_{l = 1}^{i - 1}f_{2l}(t_{2l})&\cdot \prod_{l = i}^{m - 1}f_{2l + 1}(t_{2l + 1})\cdot\prod_{k = 1}^{i - 1}\int_{t_{2k}}^{t_{2k - 2}} f_{2k - 1}(t_{2k - 1})\,dt_{2k - 1} 
			\\
			&\times \int_{t_{2i + 1}}^{t_{2i - 2}}\int_{t_{2i + 1}}^{t_{2i - 1}}f(t_{2i})f(t_{2i - 1})\,dt_{2i}\,dt_{2i - 1}
			\cdot
			\prod_{k = i + 1}^{m}\int_{t_{2k + 1}}^{t_{2k - 1}} f_{2k}(t_{2k})\,dt_{2k}.
		\end{align*}
		We see that there is only one new integral. We have 
		\begin{gather*}
			\int_{t_{2i + 1}}^{t_{2i - 2}}\int_{t_{2i + 1}}^{t_{2i - 1}}f(t_{2i})f(t_{2i - 1})\,dt_{2i}\,dt_{2i - 1} = \frac{(F(t_{2i - 2}) - F(t_{2i + 1}))^2}{2},
		\end{gather*}
		which is not greater than $32\gamma^2$ by Lemma \ref{Lemma 1}. To conclude the proof we use the same bound as in \eqref{eq: ineq for single inner integral}.
	\end{proof}
	\begin{Lem}\label{lemma: short product integrals}
		Let $f,g\in L^2([0,1])$ and $F(t) = \int_0^tf(x)\,dx$, $G(t) = \int_0^tg(x)\,dx$. Assume that $F$ and $G$ satisfy the assertions in \eqref{eq: epsilon delta assertion on the antiderivative}. Then we have
		\begin{gather*}
			|(ffgg)_t|\le 160\gamma^4,
			\qquad
			|(fgg)_t|\le 11\gamma^3,
			\qquad
			|(ffg)_t|\le 79\gamma^3.
		\end{gather*}   
	\end{Lem}
	\begin{proof}
		These bounds are sharper than the bounds in Lemma \ref{lemma: long product bound} in the degree of $\gamma$. Let us proceed  more carefully. We have
		\begin{align}
			\nonumber
			(ffgg)_t &= \int_0^tg(t_3)\cdot\left[\int_{t_3}^t\int_{t_3}^{t_1}f(t_1)f(t_2)\,dt_2\,dt_1\right]\cdot \left[\int_0^{t_3}g(t_4)\,dt_4\right]\,dt_3
			\\ 
			\nonumber
			&= \int_0^tg(t_3)\cdot\frac{(F(t) - F(t_3))^2}{2}\cdot G(t_3)\,dt_3
			\\
			\nonumber
			&= \frac{F(t)^2}{2}\int_0^tg(t_3)G(t_3)\,dt_3 - F(t)\int_0^tF(t_3)g(t_3)G(t_3)\,dt_3 
			\\
			\label{eq: ffgg bound}
			&+ \frac{1}{2}\int_0^tF^2(t_3)g(t_3)G(t_3)\,dt_3.
		\end{align}
		Let $C_F = \int_0^1 F(x)\,dx$ and $C_G = \int_0^1 G(x)\,dx$. Lemma \ref{Lemma 1} gives
		\begin{gather}
			\label{eq: tmp bounds 1}
			|C_F|\le 2\gamma,
			\qquad
			\sup_{t\in [0,1]}|F(t)|\le 4\gamma,
			\qquad
			\|F^2 - C_F^2\|_{L^2([0,1])} \le 6\eps^{1/2}\gamma,
			\\
			\label{eq: tmp bounds 2}
			|C_G|\le 2\gamma,
			\qquad
			\sup_{t\in [0,1]}|G(t)|\le 4\gamma,
			\qquad
			\|G^2 - C_G^2\|_{L^2([0,1])} \le 6\eps^{1/2}\gamma.
		\end{gather}
		For the first integral in \eqref{eq: ffgg bound} we have 
		\begin{gather*}
			\frac{F(t)^2}{2}\int_0^tg(t_3)G(t_3)\,dt_3 = \frac{F(t)^2G(t)^2}{4},
			\\
			\left|\frac{F(t)^2}{2}\int_0^tg(t_3)G(t_3)\,dt_3 \right|\le 64\gamma^4.
		\end{gather*}
		Furthermore, rewrite 
		\begin{align*}
			\int_0^tF(t_3)g(t_3)G(t_3)\,dt_3 &= C_F\int_0^tg(t_3)G(t_3)\,dt_3 + \int_0^t(F(t_3) - C_F)g(t_3)G(t_3)\,dt_3
			\\
			&= \frac{C_FG(t)^2}{2} + \int_0^t(F(t_3) - C_F)g(t_3)G(t_3)\,dt_3.
		\end{align*}
		Inequalities in \eqref{eq: epsilon delta assertion on the antiderivative} imply $\|F - C_F\|_{L^2([0,1])}\le \eps^{1/2}$ and $\|g\|_{L^2([0,1])}\le \delta$. Use this, H\"{o}lder inequality, the bounds from \eqref{eq: tmp bounds 1}, \eqref{eq: tmp bounds 2} and $\eps^{1/2}\delta\le \gamma^2$ to get
		\begin{align*}
			\left|\int_0^tF(t_3)g(t_3)G(t_3)\,dt_3\right|&\le \frac{|C_FG(t)^2|}{2} + \|F - C_F\|_{L^2}\|g\|_{L^2}\|G\|_{L^{\infty}}
			\\
			&\le\frac{(2\gamma)\cdot (4\gamma)^2}{2} + \eps^{1/2}\cdot\delta\cdot 4\gamma \le 20\gamma^3.
			% \\
			% &\le 16\gamma^3 + \eps^{1/2}\cdot \delta\cdot 4\gamma\le 20\gamma^3.
		\end{align*}
		Therefore, the estimate for the second term in \eqref{eq: ffgg bound} is
		\begin{gather*}
			\left|\frac{F(t)}{2}\int_0^tF(t_3)g(t_3)G(t_3)\,dt_3\right|\le 40\gamma^4.
		\end{gather*}
		%because $\gamma^2 = \left(\eps^{1/2} + \eps^{1/4}\delta^{1/2}\right)\ge \eps^{1/2}\delta$. 
		Similarly, for the third integral we get
		\begin{align*}
			\int_0^tF^2(t_3)g(t_3)G(t_3)\,dt_3 &= C_F^2\int_0^tg(t_3)G(t_3)\,dt_3 + \int_0^t(F(t_3)^2 - C_F^2)g(t_3)G(t_3)\,dt_3,
			\\
			\left|\int_0^tF^2(t_3)g(t_3)G(t_3)\,dt_3\right|&\le \frac{|C_F^2G(t)^2|}{2} + \|F^2 - C_F^2\|_{L^2}\|g\|_{L^2}\|G\|_{L^{\infty}}
			\\
			&\le\frac{(2\gamma)^2\cdot (4\gamma)^2}{2} + 6\eps^{1/2}\gamma\cdot\delta\cdot 4\gamma \le 56\gamma^4.
		\end{align*}
		Now the inequality $|(ffgg)_t|\le 160\gamma^4$ follows from \eqref{eq: ffgg bound}. The inequalities for $(fgg)_t$ and $(ffg)_t$ are less technical. We have
		\begin{align*}
			(fgg)_t &= \int_0^t f(t_1)\left[\int_0^{t_1}\int_0^{t_2}g(t_2)g(t_3)\,dt_3\,dt_2\right]\,dt_1 = \frac{1}{2}\int_0^t f(t_1)G(t_1)^2\,dt_1
			\\
			&=
			\frac{C_G^2}{2}\int_0^t f(t_1)\,dt_1 + \frac{1}{2}\int_0^t f(t_1)(G(t_1)^2 - C_G^2)\,dt_1,
			\\
			|(fgg)_t|&\le \frac{C_G^2|F(t)| + \|f\|_{L^2}\|G^2 - C_G^2\|_{L^2}}{2} \le \frac{(2\gamma)^2\cdot 4\gamma + \delta\cdot 6\eps^{1/2}\gamma}{2} 
			%\le 8\gamma^3 + 3\delta\eps^{1/2}\gamma
			\le 11\gamma^3.
		\end{align*}
		For $(ffg)_t$ we write
		\begin{align*}
			(ffg)_t &= \int_0^t g(t_3)\left[\int_{t_3}^{t_1}\int_{t_3}^{t}f(t_1)f(t_2)\,dt_2\,dt_1\right]\,dt_3 = \frac{1}{2}\int_0^t g(t_3)(F(t) - F(t_3))^2\,dt_3
			\\
			&=
			\frac{F(t)^2}{2}\int_0^t g(t_3)\,dt_3 - F(t)\int_0^t g(t_3)F(t_3)\,dt_3 + \frac{1}{2}\int_0^t g(t_3)F(t_3)^2\,dt_3
			\\
			&=\frac{F(t)^2G(t)}{2} - F(t)\left[C_F\int_0^t g(t_3)\,dt_3 +\int_0^tg(t_3)(F(t_3) - C_F)\,dt_3\right] + (gff)_t.
		\end{align*}
		This gives $|(ffg)_t| =\frac{(4\gamma)^2\cdot 4\gamma}{2} + 4\gamma[2\gamma\cdot 4\gamma + \delta\eps^{1/2}] + 11\gamma^3\le 79\gamma^3$. 
	\end{proof}
	If $M = (m_1m_2)$ and $N = (n_1n_2)$ are $2\times 2$ matrices with the vector-columns $m_1, m_2$ and $n_1, n_2$ we let
	\begin{gather}\label{eq: mixed det formula}
		\det(M, N) = \det((m_1n_2)) + \det((n_1m_2)).
	\end{gather}
	Equivalently, we can write
	\begin{gather*}
		\det (M, N) = \det(M + N) - \det(M) - \det(N).
	\end{gather*}
	For arbitrary $2\times 2$ matrices $Z_1,\ldots, Z_n$ we have
	\begin{gather}\label{eq: det of a sum}
		\det(Z_1 + \ldots + Z_n) = \sum_{k = 1}^n \det (Z_k) + \sum_{k = 0}^n\sum_{l = k + 1}^n \det(Z_k, Z_l).
	\end{gather}
	Let $\|\cdot\|_2$ denote the Frobenius norm of $2\times 2$ matrix, i.e., the square root of the sum of squares of entries. The following inequalities hold:
	\begin{gather}\label{eq: frobenius norm ineqs}
		|\det(M)|\le \|M\|_2^2/2,\qquad |\det(M, N)|\le \|M\|_2 \|N\|_2, \qquad \|MN\|_2\le \|M\|_2\|N\|_2.
	\end{gather}
	Formula \eqref{explicit representaion ordered exponent} is a representation of $X_A$ as a sum of $2\times 2$ matrices. Below we will substitute it into \eqref{eq: def F_A} and \eqref{eq: det of a sum}, \eqref{eq: frobenius norm ineqs} will help us estimate the value of $F_A$.
	\subsection{Proof of the theorem \ref{thm: ordered exponential}}
	
	\begin{proof}[Proof of the theorem \ref{thm: ordered exponential}]
		Let matrices $M_k(t), N_k(t)$ and $L_k$ be defined by
		\begin{gather}\label{eq: M_n def}
			X_{sA}(t) = \sum_{k\ge 0} s^kM_k(t),
			\qquad X_{sA}(t)X_{sA}(t)^T = \sum_{k\ge 0} s^kN_k(t),\qquad L_k = \int_0^1N_k(t)\,dt.
		\end{gather}
		Formula \eqref{explicit representaion ordered exponent} allows us to write out $M_k$ in terms of $A$. We have $M_0(t) = \I$ and for $k\ge 1$
		\begin{gather}\label{eq: M_k integral representation}
			M_k(t) = \int_0^t A(t_1) \int_0^{t_1} A(t_2)\int_0^{t_2}\ldots \int_0^{t_{k - 1}} A(t_k)\, dt_k\ldots dt_3\,dt_2\, dt_1.
		\end{gather}
		The definition of $N_k$ implies
		\begin{gather}\label{eq: formula for N in terms of M}
			N_k(t) = \sum_{m = 0}^k M_{m}(t)M_{k - m}^T(t).
		\end{gather}
		From \eqref{eq: def F_A} and \eqref{eq: det of a sum} we see that 
		\begin{gather*}
			F_A(s) = \det\left(\sum_{k = 0}^{\infty}s^kL_k\right) =\sum_{k = 0}^{\infty}s^{2k}\det(L_k) + \sum_{k = 0}^{\infty}\sum_{l = k + 1}^{\infty}s^{k + l}\det(L_k, L_l).
		\end{gather*}
		If we regroup the terms so that this becomes the power series in $s$, we will get
		\begin{gather}\label{eq: an representation in terms of L}
			a_{2n} = \det(L_n) + \sum_{k = 0}^{n - 1}\det(L_k,L_{2n - k}),\quad n\ge 0.
		\end{gather}
		Let us show that for $n\ge 2$ the numbers $a_{2n}$  are small. Inequality \eqref{eq: frobenius norm ineqs} implies 
		\begin{gather}\label{eq: an bound in terms of L}
			|a_{2n}|\le \|L_n\|_2^2 + \sum_{k = 0}^{n - 1}\|L_k\|_2\cdot \|L_{2n - k}\|_2.
		\end{gather}
		Every entry of $M_k$, recall \eqref{eq: M_k integral representation}, is a sum of $2^{k -1}$ integrals of the form $\pm(f_1\ldots, f_k)_t$, where $f_i = p$ or $f_i = q$ for every $1\le i \le k$. 
		Hence, by Lemma \ref{lemma: long product bound}, every entry of $M_k$ does not exceed $2^{k - 1}(8\gamma)^m$, where $m = m(k) = [(k + 1)/2]\ge k/2$. For $k\ge 1$ this gives us the inequality 
		\begin{gather}\label{eq: bound for norm Mk}
			\|M_k\|_2\le 2^k (8\gamma)^m,\qquad m = m(k) = \left[(k + 1)/2\right].
		\end{gather}
		Now formula \eqref{eq: formula for N in terms of M} yields
		\begin{align*}
			\|N_k(t)\|_2\le \sum_{l = 0}^k\|M_l(t)\|_2\cdot \|M_{k - l}(t)\|_2 &\le 2^k \sum_{l = 0}^k (8\gamma)^{m(l) + m(k - l)} 
			\\
			&\le 2^k(k + 1)(8\gamma)^{m(k)}\le 2^{2k}(8\gamma)^{m(k)},
		\end{align*}
		where we used the simple inequality $m(l) + m(k - l)\ge m(k)$ and the assertion $8\gamma \le 1$ as $\eps, \delta\to 0$. For every $k\ge 0$ we have $\|L_k\|_2^2 \le \int_0^1\|N_k(t)\|_2^2\,dt$ hence
		\begin{gather}\label{eq: L_k bound}
			\|L_k\|_2 \le 2^{2k}(8\gamma)^{m(k)}.
		\end{gather}
		Substituting this into \eqref{eq: an bound in terms of L}, we get 
		\begin{gather*}
			|a_{2n}|\le 2^{4n}(8\gamma)^{2m(n)} + \sum_{k = 0}^{n - 1} 2^{4n}(8\gamma)^{m(k) + m(2n - k)}\le (n + 1)2^{7n}\gamma^n.
		\end{gather*}
		Therefore, we get $\sum_{n\ge 4}\left|a_{2n}\right|\le \sum_{n\ge 4} (n + 1)2^{7n}\gamma^n = O(\gamma^4) = o(\eps)$ as $\eps, \delta\to 0$, recall the definition of $\gamma$ given in Lemma \ref{Lemma 1}. 
		Lemma \ref{lemma: first coefficient general situation} states $a_2 = 4\eps$, hence to conclude the proof it is left to show $a_4= o(\eps)$ and $a_6 = o(\eps)$ as $\eps,\delta\to 0$. The estimate $a_{2n} = O(\gamma^n)$ for $n = 2,3$ gives $a_4 = O(\gamma^2)$ and $a_6 = O(\gamma^3)$ respectively, which is not strong enough. 
		For $n = 3$ it improves with more careful consideration of the terms in \eqref{eq: an representation in terms of L}. To deal with $n = 2$ we explicitly write out the representation of $a_4$ in terms of the functions $p$ and $q$, see \eqref{eq: final form of a4} below. The $a_4$ part is more technical so we proceed with the estimate of $a_6$. 
		Equation \eqref{eq: an representation in terms of L} for $n = 3$ becomes
		\begin{align}\label{eq: a_6 representation}
			a_6 = \det(L_3) + \det(L_1 , L_{5})+ 
			\det(L_2 , L_{4}) + \det(L_0 , L_{6}).
		\end{align}
		From \eqref{eq: L_k bound} and \eqref{eq: frobenius norm ineqs} we have
		\begin{gather}
			\nonumber
			%|\det(L_3)|\le \|L_3\|_2^2\le 2^64^2(8\gamma)^{2m(3)} = 2^{17}\gamma^4 = o(\eps),
			|\det(L_3)|\le \|L_3\|_2^2 = O\left(\gamma^{2m(3)}\right) = O(\gamma^4) = o(\eps), \quad \eps,\delta\to 0,
			\\
			\label{eq: o(eps) bound for lk}
			\left|\det(L_1 , L_{5})\right|\le \|L_1\|_2\cdot \|L_5\|_2 = O\left(\gamma^{m(1) + m(5)}\right)  = O(\gamma^4) = o(\eps), \quad \eps,\delta\to 0.
		\end{gather}
		Furthermore, we have $L_0 = N_0 = \I$ hence
		\begin{gather}\label{eq: det L0 + L6}
			\det(L_0 , L_{6}) = \trace(L_6) = \int_0^1\trace(N_6(t))\,dt.
		\end{gather}
		Rewrite $N_6$ using formula \eqref{eq: formula for N in terms of M}:
		\begin{align}\label{eq: trace N_6}
			\trace(N_6)
			&= \trace\left(\sum_{k = 0}^6 M_kM^T_{6 - k}\right)
			\\
			\nonumber
			&= \trace(M_6 + M_6^T + M_2M_4^T + M_4M_2^T) + \trace(M_3 M_3^T + M_1M_5^T + M_5M_1^T).
		\end{align}
		Similarly to \eqref{eq: o(eps) bound for lk}, \eqref{eq: bound for norm Mk} implies
		\begin{gather}\label{eq: trace first part}
			\trace(M_3 M_3^T + M_1M_5^T + M_5M_1^T) = O(\gamma^4) = o(\eps), \quad \eps,\delta\to 0.
		\end{gather}
		Consider the matrix-valued function
		\begin{gather*}
			K(t_1, t_2) = A(t_1)A(t_2)= 
			\begin{pmatrix}
				p(t_1)p(t_2) +  q(t_1) q(t_2)  & p(t_1)q(t_2) - q(t_1)p(t_2)
				\\
				-p(t_1)q(t_2) + q(t_1)p(t_2) & p(t_1)p(t_2) +  q(t_1) q(t_2)
			\end{pmatrix}.
		\end{gather*}
		Formula \eqref{eq: M_k integral representation} for $ k= 2,4,6$ reads as
		\begin{gather*}
			M_2 = \int_0^t\int_0^{t_1} K(t_1,t_2)\,dt_2\,dt_1, \quad M_4 = \int_0^t\ldots\int_0^{t_3} K(t_1,t_2)K(t_3,t_4)\,dt_4\ldots\,dt_1,
			\\
			M_6 = \int_0^t\ldots\int_0^{t_5} K(t_1,t_2)K(t_3,t_4)K(t_5,t_6)\,dt_6\ldots\,dt_1.
		\end{gather*}
		Every entry of $M_6(t)$ is a sum of a $32$ integrals of the form $\pm(f_1\ldots f_6)_t$ where $f_k = p$ or $f_k = q$ for every $1\le k \le 6$; the entries of $M_2M_4^T$ are the similar sums of the terms $(f_1f_2)_t(f_3\ldots f_6)_t$. By Lemma \ref{lemma: long product bound}, if $f_1 = f_2$ or $f_3 = f_4$ or $f_5 = f_6$, then the corresponding term is $O(\gamma^4)= o(\eps)$ as $\eps,\delta\to 0$. Therefore
		\begin{gather*}
			M_6(t) = o(\eps) + \int_0^t\ldots\int_0^{t_5} 
			\tilde K(t_1,t_2)\tilde K(t_3,t_4)\tilde K(t_5,t_6)\,dt_6\ldots\,dt_1,
			\\
			M_2M_4^T =  o(\eps) + \int_0^t\int_0^{t_1} K(t_1,t_2)\,dt_2\,dt_1 \int_0^t\ldots\int_0^{t_3} K(t_1,t_2)K(t_3,t_4)\,dt_4\ldots\,dt_1,
		\end{gather*}
		where the matrix $\tilde K$ is defined by
		\begin{gather*}
			\tilde K(t_1, t_2) = 
			\begin{pmatrix}
				0  & p(t_1)q(t_2) - q(t_1)p(t_2)
				\\
				-p(t_1)q(t_2) + q(t_1)p(t_2) & 0
			\end{pmatrix}.
		\end{gather*}
		Notice that $\tilde K + \tilde K^T = 0$ hence the integrals in the equation above also satisfy similar property and 
		\begin{gather*}
			\trace(M_6 + M_6^T + M_2M_4^T + M_4M_2^T) = o(\eps),\quad\eps,\delta\to 0.
		\end{gather*}
		Substitution of this and \eqref{eq: trace first part} into \eqref{eq: trace N_6} gives $\trace(N_6) = o(\eps)$. Now \eqref{eq: det L0 + L6} and \eqref{eq: a_6 representation} imply
		\begin{gather}
			a_6 = \det(L_2 , L_{4}) + o(\eps), \quad \eps,\delta\to 0.
		\end{gather}
		We have $|\det(L_2 , L_{4})|\le \|L_2\|_2\cdot \|L_4\|_2$ and, by \eqref{eq: L_k bound}, $\|L_4\|_2 = O(\gamma^2)$. Thus the estimate $a_6 = o(\eps)$ will immediately follow from
		\begin{gather}\label{eq: L^2 estimate}
			\|L_2\|_2 = O(\gamma^2), \quad \eps,\delta\to 0.
		\end{gather}
		Let us calculate $L_2$. We have 
		\begin{gather}
			\label{eq: M_0, M_1, M_2 formula}
			M_0= \I,
			% \begin{pmatrix}
				% 1  &  0
				% \\
				% 0 &  1 
				% \end{pmatrix},
			\quad
			M_1=
			\begin{pmatrix}
				-(q)_t   &  (p)_t 
				\\
				(p)_t   &  (q)_t 
			\end{pmatrix},
			\quad
			M_2=
			\begin{pmatrix}
				(pp)_t + (qq)_t &  (pq)_t - (qp)_t
				\\
				-(pq)_t + (qp)_t   &  (pp)_t + (qq)_t
			\end{pmatrix}.
		\end{gather}
		Formula \eqref{eq: formula for N in terms of M} for $n = 1,2,3$ becomes
		\begin{gather}
			\nonumber
			N_0 = \I,
			% \begin{pmatrix}
				%     1  &  0
				%     \\
				%     0 &  1 
				% \end{pmatrix},
			\quad
			N_1 = M_0M_1^T + M_1M_0^T = 2M_1,
			% \begin{pmatrix}
				% -(q)_t   &  (p)_t 
				% \\
				% (p)_t   &  (q)_t 
				% \end{pmatrix},
			\\
			\nonumber
			M_0M_2^T = M_2^T
			,
			\quad 
			M_1M_1^T = ((p)_t^2 + (q)_t^2)\I = 2((pp)_t + (qq)_t)\I,
			\\
			\label{eq: N_0, N_1, N_2 formula}
			N_2 = M_0M_2^T + M_1M_1^T + M_2M_0^T =
			4((pp)_t + (qq)_t)\I.
		\end{gather}
		By Lemma \ref{lemma: long product bound} we know  $(pp)_t = O(\gamma^2)$ and $(qq)_t = O(\gamma^2)$, which implies \eqref{eq: L^2 estimate} and finishes the $a_6$ part.
		\medskip
		
		Calculation of $a_4$ is more technical. We will show that
		\begin{align}
			\nonumber
			a_4/16 &= 2\int_0^1(qqqq)_t\,dt - 2 \int_0^1(qqq)_t\,dt\cdot \int_0^1(q)_t\,dt + \left(\int_0^1(qq)_t\,dt\right)^2
			\\ \nonumber
			&+ 2\int_0^1(pppp)_t\,dt - 2 \int_0^1(ppp)_t\,dt\cdot \int_0^1(p)_t\,dt + \left(\int_0^1(pp)_t\,dt\right)^2
			\\ \nonumber
			&+ 2\int_0^1(qqpp)_t\,dt + 2\int_0^1(ppqq)_t\,dt + 2\int_0^1(qq)_t\,dt\cdot \int_0^1(pp)_t\,dt
			\\ \label{eq: final form of a4}
			&-2\int_0^1(q)_t\,dt\cdot \int_0^1(qpp)_t\,dt - 2\int_0^1(p)_t\,dt\cdot \int_0^1(pqq)_t\,dt.
		\end{align}
		Applications of Lemmas \ref{lemma: long product bound} and \ref{lemma: short product integrals} to all the integrals immediately give the bound $a_4 = o(\eps)$. Thus to finish the proof of the theorem we need to establish \eqref{eq: final form of a4}. 
		Formula \eqref{eq: an representation in terms of L} for $n = 2 $ reads as 
		\begin{align}
			\label{eq: a4 in terms of L}
			a_4 = \det(L_2) + \det(L_1 , L_{3}) + \det(L_0 , L_{4}).
		\end{align}
		Let us calculate matrices $L_0,\ldots, L_4$. We calculated matrices $M_0, M_1$ and $M_2$ previously in \eqref{eq: M_0, M_1, M_2 formula}. We have 
		\begin{gather*}
			M_3=
			\begin{pmatrix}
				-(qqq)_t - (qpp)_t + (pqp)_t - (ppq)_t &  (qqp)_t  -(qpq)_t + (pqq)_t + (ppp)_t
				\\
				(qqp)_t  -(qpq)_t + (pqq)_t + (ppp)_t   &  (qqq)_t + (qpp)_t - (pqp)_t + (ppq)_t
			\end{pmatrix}.
		\end{gather*}
		The matrix $M_4(t)$ is of the form $\smatrix{x}{y}{-y}{x}$, its first column is given by
		\begin{gather*}
			\begin{pmatrix}
				(qqqq)_t + (qqpp)_t - (qpqp)_t + (qppq)_t + (pqqp)_t - (pqpq)_t + (ppqq)_t + (pppp)_t  &  \ldots
				\\
				-(qqqp)_t + (qqpq)_t - (qpqq)_t - (qppp)_t + (pqqq)_t + (pqpp)_t - (ppqp)_t + (pppq)_t  &  \ldots
			\end{pmatrix}.
		\end{gather*}
		Matrices $N_0, N_1$ and $N_2$ are calculated in \eqref{eq: N_0, N_1, N_2 formula}. Let us write out $N_3, N_4$ using the relation \eqref{eq: formula for N in terms of M}. We have $M_0M_3^T = M_3^T$; $M_1M_2^T$ is of the form $\smatrix{x}{y}{y}{-x}$ and 
		\begin{gather*}
			M_1M_2^T = 
			\begin{pmatrix}
				-(q)_t\cdot (qq)_t - (q)_t\cdot (pp)_t - (p)_t\cdot (qp)_t + (p)_t\cdot (pq)_t & \ldots 
				\\
				(p)_t\cdot (qq)_t + (p)_t\cdot (pp)_t -(q)_t\cdot (qp)_t + (q)_t\cdot (pq)_t   & \ldots
			\end{pmatrix}.    
		\end{gather*}
		From the definition \eqref{eq: short integral def} of $(\ldots)_t$ we see that
		\begin{gather}\label{eq: product of t integrals}
			(f)_t\cdot (g_1g_2\ldots g_n)_t = (fg_1g_2\ldots g_n)_t + (g_1fg_2\ldots g_n)_t + \ldots + (g_1g_2\ldots g_nf)_t.  
		\end{gather}
		Therefore the expression for $M_1M_2^T$ rewrites as 
		\begin{gather*}
			\begin{pmatrix}
				-3(qqq)_t - [(qpp)_t  + (pqp)_t+ (ppq)_t] - [(pqp)_t+ 2(qpp)_t] + [2(ppq)_t + (pqp)_t]& \ldots 
				\\
				[(pqq)_t + (qpq)_t + (qqp)_t] + 3(ppp)_t - [2(qqp)_t + (qpq)_t]  + [(qpq)_t + 2(pqq)_t] &  \ldots
			\end{pmatrix}
			\\
			= \begin{pmatrix}
				-3(qqq)_t   - 3(qpp)_t + (ppq)_t - (pqp)_t& \ldots 
				\\
				3(pqq)_t + 3(ppp)_t - (qqp)_t + (qpq)_t  &\ldots
			\end{pmatrix}.
		\end{gather*}
		It follows that 
		\begin{gather*}
			N_3 = M_3 + M_3^T + M_1M_2^T + M_2M_1^T=
			\begin{pmatrix}
				-8(qqq)_t   - 8(ppq)_t& 8(qqp)_t + 8(ppp)_t
				\\
				8(qqp)_t + 8(ppp)_t & 8(qqq)_t   + 8(ppq)_t
			\end{pmatrix}.
		\end{gather*}
		Similarly to \eqref{eq: det L0 + L6} we have
		\begin{gather}\label{eq: det L0L4}
			\det(L_0 , L_{4}) = \trace(L_4) = \int_0^t\trace(N_4(t))\,dt.
		\end{gather}
		Formula \eqref{eq: formula for N in terms of M} for $n = 4$ is
		\begin{gather}\label{eq: formula N4}
			N_4 = M_4 + M_4^T + M_1M_3^T + M_3M_1^T + M_2M_2^T.
		\end{gather}
		Notice that all of the matrices are of the form $\smatrix{x}{y}{-y}{x}$ hence to find the trace we only need to calculate the upper-left element. For $M_1M_3^T$ it equals
		\begin{gather*}
			-(q)_t \left(-(qqq)_t - (qpp)_t + (pqp)_t - (ppq)_t \right) + (p)_t\left((qqp)_t  -(qpq)_t + (pqq)_t + (ppp)_t\right).
		\end{gather*}
		By formula \eqref{eq: product of t integrals} it rewrites as
		\begin{align*}
			4(qqqq)_t &+ [2(qqpp)_t + (qpqp)_t + (qppq)_t] 
			\\
			&- [(qpqp)_t + 2(pqqp)_t + (pqpq)_t] 
			+[(qppq)_t + (pqpq)_t + 2(ppqq)_t]
			\\
			&+ [(pqqp)_t + (qpqp)_t + 2(qqpp)_t] - [(pqpq)_t + 2(qppq)_t + (qpqp)_t] 
			\\
			&+ [2(ppqq)_t + (pqpq)_t + (pqqp)_t] + 4(pppp)_t
			\\
			&= 4(qqqq)_t + 4(qqpp)_t + 4(ppqq)_t + 4(pppp)_t.
		\end{align*}
		Further, for $M_2M_2^T$ we get
		\begin{align*}
			((qq)_t + (pp)_t)^2 &+ ((pq)_t - (qp)_t)^2 
			\\
			&= (qq)_t^2 + 2(qq)_t(pp)_t + (pp)_t^2 + (pq)_t^2 + (qp)_t^2 - 2(pq)_t(qp)_t .
		\end{align*}
		Similarly to \eqref{eq: product of t integrals} we notice that
		\begin{align*}
			(f_1f_2)_t(g_1g_2)_t = (f_1f_2g_1g_2)_t+ (f_1g_1f_2g_2)_t &+ (f_1g_1g_2f_2)_t + (g_1f_1f_2g_2)_t 
			\\
			&+ (g_1f_1g_2f_2)_t + (g_1g_2f_1f_2)_t.
		\end{align*}
		Hence the expression for $M_2M_2^T$ takes the form
		\begin{gather*}
			6(qqqq)_t + 2[(qqpp)_t + (qpqp)_t + (qppq)_t + (pqqp)_t + (pqpq)_t + (ppqq)_t] 
			+ 6(pppp)_t
			\\
			+[2(qpqp)_t + 4(qqpp)_t] + [2(pqpq)_t + 4(ppqq)_t]
			\\
			- 2[2(qppq)_t + (qpqp)_t + (pqpq)_t + 2(pqqp)_t] 
			\\
			= 6(qqqq)_t + 6(qqpp)_t + 2(qpqp)_t - 2(qppq)_t - 2(pqqp)_t + 2 (pqpq)_t + 6(ppqq)_t + 6(pppp)_t.
		\end{gather*}
		When we sum $M_4 + M_2M_2^T + M_4^T$ the terms $(qpqp)_t, 2(qppq)_t, 2(pqqp)_t ,2 (pqpq)_t$ cancel out hence by \eqref{eq: formula N4} and \eqref{eq: det L0L4} we get $N_4 = 16((qqqq)_t + (qqpp)_t + (ppqq)_t + (pppp)_t)\I$ and
		\begin{gather}
			% \nonumber
			% N_4 = 16((qqqq)_t + (qqpp)_t + (ppqq)_t + (pppp)_t)\I,
			% \\
			\label{eq: a_4 part L_0, L_4}
			\det(L_0 , L_{4}) = 32\int_0^1[(qqqq)_t + (qqpp)_t + (ppqq)_t + (pppp)_t]\,dt.
		\end{gather}
		Next, we write
		\begin{align}
			\nonumber
			\det L_2 &= \det\left(\int_0^1N_2(t)\,dt\right)
			= \det\left(4\int_0^1((qq)_t + (pp)_t)\I\,dt\right)
			\\
			\nonumber
			&=16\left(\int_0^1 (qq)_t\,dt + \int_0^1 (pp)_t\,dt\right)^2 
			\\
			\label{eq: a_4 part L_2}
			&= 16 \left(\int_0^1 (qq)_t\,dt\right)^2 + 16 \left(\int_0^1 (pp)_t\,dt\right)^2 + 32 \int_0^1 (qq)_t\,dt\int_0^1 (pp)_t\,dt.
		\end{align}
		The final part is $\det(L_1 , L_{3})$ which we estimate via \eqref{eq: mixed det formula}:
		\begin{gather*}
			\det(L_1 , L_{3}) =\det
			\begin{pmatrix}
				-2\int_0^1(q)_t\,dt   & 8\int_0^1 (qqp)_t + (ppp)_t  \,dt
				\\
				2\int_0^1(p)_t\,dt   &  8\int_0^1 (qqq)_t  + (ppq)_t\,dt
			\end{pmatrix}
			\\
			+
			\det
			\begin{pmatrix}
				-8\int_0^1 (qqq)_t  + (ppq)_t\,dt& 2\int_0^1(p)_t\,dt
				\\
				8\int_0^1 (qqp)_t + (ppp)_t \,dt& 2\int_0^1(q)_t \,dt
			\end{pmatrix}
			\\
			= -32\int_0^1(q)_t\,dt\cdot \int_0^1((qqq)_t + (ppq)_t)\,dt - 32\int_0^1(p)_t\,dt\cdot \int_0^1((ppp)_t + (qqp)_t)\,dt.
		\end{gather*}
		Formula \eqref{eq: final form of a4} follows by substituting the last equality, \eqref{eq: a_4 part L_0, L_4} and \eqref{eq: a_4 part L_2} into \eqref{eq: a4 in terms of L}. The proof of the theorem is concluded.
	\end{proof}
	
	\bibliographystyle{plain}
	\bibliography{ref}
	
\end{document}